\documentclass[11pt]{amsart}

\usepackage{epigamath}
\usepackage{microtype}

\usepackage[english]{babel}

\numberwithin{equation}{section}

\titlespacing*{\subsubsection}
{\parindent}{*2}{\wordsep}
\titleformat{\subsubsection}[runin]{\normalfont \normalsize \bfseries
\justifyheading}{\thesubsubsection.}{0em}{}

\usepackage{textcmds} 
\usepackage{amsmath}
\usepackage{amsxtra}
\usepackage{amscd}
\usepackage{amsthm}
\usepackage{amsfonts}
\usepackage{amssymb}
\usepackage{eucal}
\usepackage[all]{xy}
\usepackage{graphicx}
\usepackage{comment}
\usepackage{epsfig}
\usepackage{psfrag}
\usepackage{mathrsfs}
\usepackage{amscd}
\usepackage{rotating}
\usepackage{lscape}
\usepackage{amsbsy}
\usepackage{verbatim}
\usepackage{moreverb}
\usepackage{url}
\usepackage{extarrows}

\usepackage{cancel}
\usepackage{bbm}

\usepackage{tikz-cd}
\usepackage{tikz}
\usetikzlibrary{matrix,arrows,decorations.pathmorphing}

\DeclareMathAlphabet{\mathpzc}{OT1}{pzc}{m}{it}

\newtheorem{cor}[subsubsection]{Corollary}
\newtheorem{lem}[subsubsection]{Lemma}
\newtheorem{prop}[subsubsection]{Proposition}

\newtheorem{conj}[subsubsection]{Conjecture}
\newtheorem{thm}[subsubsection]{Theorem}

\theoremstyle{remark}
\newtheorem{rem}[subsubsection]{Remark}
\newtheorem{example}[subsubsection]{Example}

\newcommand{\nc}{\newcommand}
\nc{\renc}{\renewcommand}
\nc{\ssec}{\subsection}
\nc{\sssec}{\subsubsection}
\nc{\on}{\operatorname}

\nc\ol{\overline}
\nc\wt{\widetilde}
\nc\tboxtimes{\wt{\boxtimes}}
\nc\tstar{\wt{\star}}
\nc{\alp}{\alpha}

\nc{\ZZ}{{\mathbb Z}}
\nc{\NN}{{\mathbb N}}
\nc{\OO}{{\mathbb O}}
\renc{\SS}{{\mathbb S}}
\nc{\DD}{{\mathbb D}}
\nc{\GG}{{\mathbb G}}
\renewcommand{\AA}{{\mathbb A}}
\nc{\Fq}{{\mathbb F}_q}
\nc{\Fqb}{\ol{{\mathbb F}_q}}
\nc{\Ql}{\ol{{\mathbb Q}_\ell}}
\nc{\id}{\text{id}}
\nc\X{\mathcal X}

\nc{\Hom}{\on{Hom}}
\nc{\Lie}{\on{Lie}}
\nc{\Loc}{\on{Loc}}
\nc{\Pic}{\on{Pic}}
\nc{\Bun}{\on{Bun}}
\nc{\IC}{\on{IC}}
\nc{\Aut}{\on{Aut}}
\nc{\rk}{\on{rk}}
\nc{\Sh}{\on{Sh}}
\nc{\Perv}{\on{Perv}}
\nc{\pos}{{\on{pos}}}
\nc{\Conv}{\on{Conv}}
\nc{\Sph}{\on{Sph}}
\nc{\Sym}{\on{Sym}}
\nc{\BunBb}{\overline{\Bun}_B}
\nc{\BunNb}{\overline{\Bun}_N}
\nc{\BunTb}{\overline{\Bun}_T}
\nc{\BunBbm}{\overline{\Bun}_{B^-}}
\nc{\BunBbel}{\overline{\Bun}_{B,el}}
\nc{\BunBbmel}{\overline{\Bun}_{B^-,el}}
\nc{\Buno}{\overset{o}{\Bun}}
\nc{\BunPb}{{\overline{\Bun}_P}}
\nc{\BunBM}{\Bun_{B(M)}}
\nc{\BunBMb}{\overline{\Bun}_{B(M)}}
\nc{\BunPbw}{{\widetilde{\Bun}_P}}
\nc{\BunBP}{\widetilde{\Bun}_{B,P}}
\nc{\GUb}{\overline{G/U}}
\nc{\GUPb}{\overline{G/U(P)}}

\nc\syminfty{\on{Sym}^{\infty}}
\nc\lal{\ol{\lambda}}
\nc\xl{\ol{x}}
\nc\thl{\ol{\theta}}
\nc\nul{\ol{\nu}}
\nc\mul{\ol{\mu}}
\nc\Sum\Sigma
\nc{\oX}{\overset{\circ}{X}{}}
\nc{\hl}{\overset{\leftarrow}h{}}
\nc{\hr}{\overset{\rightarrow}h{}}
\nc{\M}{{\mathcal M}}
\nc{\N}{{\mathcal N}}
\nc{\F}{{\mathcal F}}
\renc{\D}{{\mathcal D}}
\nc{\Y}{{\mathcal Y}}
\nc{\G}{{\mathcal G}}
\nc{\E}{{\mathcal E}}
\nc{\CalC}{{\mathcal C}}
\nc\Dh{\widehat{\D}}
\renewcommand{\O}{{\mathcal O}}
\nc{\K}{{\mathcal K}}
\renewcommand{\S}{\bbS}
\nc{\T}{{\mathcal T}}
\nc{\V}{{\mathcal V}}
\renc{\P}{{\mathcal P}}
\nc{\A}{{\AA}}
\nc{\B}{{\BB}}
\nc{\U}{{\mathcal U}}

\nc{\frn}{{\check{\mathfrak u}(P)}}

\nc{\fC}{\mathfrak C}
\nc\f{{\mathfrak f}}

\nc{\qo}{{\mathfrak q}}
\nc{\po}{{\mathfrak p}}
\nc{\s}{{\mathfrak s}}
\nc\w{\text{w}}
\renewcommand{\r}{{\mathfrak r}}

\renewcommand{\mod}{{\on{-}\mathsf{mod}}}

\nc\mathi\iota
\nc\Spec{\on{Spec}}
\nc\Mod{\on{Mod}}
\nc{\tw}{\widetilde{\mathfrak t}}
\nc{\pw}{\widetilde{\mathfrak p}}
\nc{\qw}{\widetilde{\mathfrak q}}
\nc{\jw}{\widetilde j}

\nc{\grb}{\overline{\Gr_{X^{\fset}}}}
\nc{\I}{\mathcal I}
\renewcommand{\i}{\mathfrak i}
\renewcommand{\j}{\mathfrak j}

\nc{\lambdach}{{\check\lambda}}
\nc{\Lambdach}{{\check\Lambda}{}}
\nc{\much}{{\check\mu}}
\nc{\omegach}{{\check\omega}}
\nc{\nuch}{{\check\nu}}
\nc{\etach}{{\check\eta}}
\nc{\alphach}{{\check\alpha}}

\nc{\rhoch}{{\check\rho}}

\nc{\Hb}{\overline{\H}}

\nc{\BA}{{\mathbb{A}}}
\nc{\BB}{\mathbb{B}}
\nc{\BC}{{\mathbb{C}}}
\nc{\BD}{{\mathbb{D}}}
\nc{\BE}{{\mathbb{E}}}
\nc{\BF}{{\mathbb{F}}}
\nc{\BG}{{\mathbb{G}}}
\nc{\BH}{{\mathbb{H}}}
\nc{\BI}{{\mathbb{I}}}
\nc{\BM}{{\mathbb{M}}}
\nc{\BN}{{\mathbb{N}}}
\nc{\BO}{{\mathbb{O}}}
\nc{\BP}{{\mathbb{P}}}
\nc{\BQ}{{\mathbb{Q}}}
\nc{\BR}{{\mathbb{R}}}
\nc{\BS}{{\mathbb{S}}}
\nc{\BT}{{\mathbb{T}}}
\nc{\BV}{{\mathbb{V}}}
\nc{\BZ}{{\mathbb{Z}}}

\nc{\bbone}{\mathbbm{1}}
\nc{\bbA}{{\mathbb{A}}}
\nc{\bbB}{\mathbb{B}}
\nc{\bbC}{{\mathbb{C}}}
\nc{\bbD}{{\mathbb{D}}}
\nc{\bbE}{{\mathbb{E}}}
\nc{\bbF}{{\mathbb{F}}}
\nc{\bbG}{{\mathbb{G}}}
\nc{\bbH}{{\mathbb{H}}}
\nc{\bbI}{{\mathbb{I}}}
\nc{\bbL}{{\mathbb{L}}}
\nc{\bbM}{{\mathbb{M}}}
\nc{\bbN}{{\mathbb{N}}}
\nc{\bbO}{{\mathbb{O}}}
\nc{\bbP}{{\mathbb{P}}}
\nc{\bbQ}{{\mathbb{Q}}}
\nc{\bbR}{{\mathbb{R}}}
\nc{\bbS}{{\mathbb{S}}}
\nc{\bbT}{{\mathbb{T}}}
\nc{\bbU}{{\mathbb{U}}}
\nc{\bbV}{{\mathbb{V}}}
\nc{\bbW}{{\mathbb{W}}}
\nc{\bbX}{{\mathbb{X}}}
\nc{\bbY}{{\mathbb{Y}}}
\nc{\bbZ}{{\mathbb{Z}}}

\nc{\CA}{{\mathcal{A}}}
\nc{\CB}{{\mathcal{B}}}

\nc{\CE}{{\mathcal{E}}}
\nc{\CF}{{\mathcal{F}}}
\nc{\CH}{{\mathcal{H}}}

\nc{\CL}{{\mathcal{L}}}
\nc{\CC}{{\mathcal{C}}}
\nc{\CG}{{\mathcal{G}}}
\nc{\CM}{{\mathcal{M}}}
\nc{\CN}{{\mathcal{N}}}
\nc{\CK}{{\mathcal{K}}}
\nc{\CO}{{\mathcal{O}}}
\nc{\CP}{{\mathcal{P}}}
\nc{\CQ}{{\mathcal{Q}}}
\nc{\CR}{{\mathcal{R}}}
\nc{\CS}{{\mathcal{S}}}
\nc{\CU}{{\mathcal{U}}}
\nc{\CV}{{\mathcal{V}}}
\nc{\CW}{{\mathcal{W}}}
\nc{\CX}{{\mathcal{X}}}
\nc{\CY}{{\mathcal{Y}}}
\nc{\CZ}{{\mathcal{Z}}}
\nc{\CI}{{\mathcal{I}}}

\nc{\csM}{{\check{\mathcal A}}{}}
\nc{\oM}{{\overset{\circ}{\mathcal M}}{}}
\nc{\obM}{{\overset{\circ}{\mathbf M}}{}}
\nc{\oCA}{{\overset{\circ}{\mathcal A}}{}}
\nc{\obA}{{\overset{\circ}{\mathbf A}}{}}
\nc{\ooM}{{\overset{\circ}{M}}{}}
\nc{\osM}{{\overset{\circ}{\mathsf M}}{}}
\nc{\vM}{{\overset{\bullet}{\mathcal M}}{}}
\nc{\nM}{{\underset{\bullet}{\mathcal M}}{}}
\nc{\oD}{{\overset{\circ}{\mathcal D}}{}}
\nc{\obC}{{\overset{\circ}{\mathbf C}}{}}
\nc{\obD}{{\overset{\circ}{\mathbf D}}{}}
\nc{\oA}{{\overset{\circ}{\mathbb A}}{}}
\nc{\op}{{\overset{\bullet}{\mathbf p}}{}}
\nc{\oU}{{\overset{\bullet}{\mathcal U}}{}}
\nc{\oZ}{{\overset{\circ}{\mathcal Z}}{}}
\nc{\ofZ}{{\overset{\circ}{\mathfrak Z}}{}}
\nc{\oF}{{\overset{\circ}{\fF}}}

\nc{\fa}{{\mathfrak{a}}}
\nc{\fb}{{\mathfrak{b}}}
\nc{\fc}{{\mathfrak{c}}}
\nc{\fd}{{\mathfrak{d}}}
\nc{\ff}{{\mathfrak{f}}}
\nc{\fg}{{\mathfrak{g}}}
\nc{\fgl}{{\mathfrak{gl}}}
\nc{\fh}{{\mathfrak{h}}}
\nc{\fj}{{\mathfrak{j}}}
\nc{\fl}{{\mathfrak{l}}}
\nc{\fm}{{\mathfrak{m}}}
\nc{\fn}{{\mathfrak{n}}}
\nc{\fu}{{\mathfrak{u}}}
\renc{\u}{\fu}
\nc{\fp}{{\mathfrak{p}}}
\nc{\fr}{{\mathfrak{r}}}
\nc{\fs}{{\mathfrak{s}}}
\nc{\ft}{{\mathfrak{t}}}
\nc{\fz}{{\mathfrak{z}}}
\nc{\fsl}{{\mathfrak{sl}}}
\nc{\hsl}{{\widehat{\mathfrak{sl}}}}
\nc{\hgl}{{\widehat{\mathfrak{gl}}}}
\nc{\hg}{{\widehat{\mathfrak{g}}}}
\nc{\chg}{{\widehat{\mathfrak{g}}}{}^\vee}
\nc{\hn}{{\widehat{\mathfrak{n}}}}
\nc{\chn}{{\widehat{\mathfrak{n}}}{}^\vee}

\nc{\fA}{{\mathfrak{A}}}
\nc{\fB}{{\mathfrak{B}}}
\nc{\fD}{{\mathfrak{D}}}
\nc{\fE}{{\mathfrak{E}}}
\nc{\fF}{{\mathfrak{F}}}
\nc{\fG}{{\mathfrak{G}}}
\nc{\fK}{{\mathfrak{K}}}
\nc{\fL}{{\mathfrak{L}}}
\nc{\fM}{{\mathfrak{M}}}
\nc{\fN}{{\mathfrak{N}}}
\nc{\fP}{{\mathfrak{P}}}
\nc{\fU}{{\mathfrak{U}}}
\nc{\fV}{{\mathfrak{V}}}
\nc{\fX}{{\mathfrak{X}}}
\nc{\fY}{{\mathfrak{Y}}}
\nc{\fZ}{{\mathfrak{Z}}}

\nc{\bb}{{\mathbf{b}}}
\nc{\bc}{{\mathbf{c}}}
\nc{\bd}{{\mathbf{d}}}
\nc{\bbf}{{\mathbf{f}}}
\nc{\be}{{\mathbf{e}}}
\nc{\bg}{{\mathbf{g}}}
\nc{\bi}{{\mathbf{i}}}
\nc{\bj}{{\mathbf{j}}}
\nc{\bn}{{\mathbf{n}}}
\nc{\bo}{{\mathbf{o}}}
\nc{\bp}{{\mathbf{p}}}
\nc{\bq}{{\mathbf{q}}}
\nc{\bt}{{\mathbf{t}}}
\nc{\bu}{{\mathbf{u}}}
\nc{\bv}{{\mathbf{v}}}
\nc{\bx}{{\mathbf{x}}}
\nc{\bs}{{\mathbf{s}}}
\nc{\by}{{\mathbf{y}}}
\nc{\bw}{{\mathbf{w}}}
\nc{\bA}{{\mathbf{A}}}
\nc{\bK}{{\mathbf{K}}}
\nc{\bB}{{\mathbf{B}}}
\nc{\bC}{{\mathbf{C}}}
\nc{\bG}{{\mathbf{G}}}
\nc{\bD}{{\mathbf{D}}}
\nc{\bH}{{\mathbf{H}}}
\nc{\bM}{{\mathbf{M}}}
\nc{\bN}{{\mathbf{N}}}
\nc{\bO}{{\mathbf{O}}}
\nc{\bT}{{\mathbf{T}}}
\nc{\bV}{{\mathbf{V}}}
\nc{\bW}{{\mathbf{W}}}
\nc{\bX}{{\mathbf{X}}}
\nc{\bZ}{{\mathbf{Z}}}
\nc{\bS}{{\mathbf{S}}}

\nc{\sA}{{\mathsf{A}}}
\nc{\sB}{{\mathsf{B}}}
\nc{\sC}{{\mathsf{C}}}
\nc{\sD}{{\mathsf{D}}}
\nc{\sF}{{\mathsf{F}}}
\nc{\sG}{{\mathsf{G}}}
\nc{\sK}{{\mathsf{K}}}
\nc{\sM}{{\mathsf{M}}}
\nc{\sO}{{\mathsf{O}}}
\nc{\sW}{{\mathsf{W}}}
\nc{\sQ}{{\mathsf{Q}}}
\nc{\sP}{{\mathsf{P}}}
\nc{\sV}{{\mathsf{V}}}
\nc{\sS}{{\mathsf{S}}}
\nc{\sT}{{\mathsf{T}}}
\nc{\sZ}{{\mathsf{Z}}}
\nc{\sfp}{{\mathsf{p}}}
\nc{\sll}{{\mathsf{l}}}
\nc{\sr}{{\mathsf{r}}}
\nc{\bk}{{\mathsf{k}}}
\nc{\sg}{{\mathsf{g}}}

\nc{\sff}{{\mathsf{f}}}
\nc{\sfb}{{\mathsf{b}}}
\nc{\sfc}{{\mathsf{c}}}
\nc{\sd}{{\mathsf{d}}}
\nc{\se}{{\mathsf{e}}}

\nc{\BK}{{\bar{K}}}

\nc{\tA}{{\widetilde{\mathbf{A}}}}
\nc{\tB}{{\widetilde{\mathcal{B}}}}
\nc{\tg}{{\widetilde{\mathfrak{g}}}}
\nc{\tG}{{\widetilde{G}}}
\nc{\TM}{{\widetilde{\mathbb{M}}}{}}
\nc{\tO}{{\widetilde{\mathsf{O}}}{}}
\nc{\tU}{{\widetilde{\mathfrak{U}}}{}}
\nc{\TZ}{{\tilde{Z}}}
\nc{\tx}{{\tilde{x}}}
\nc{\tbv}{{\tilde{\bv}}}
\nc{\tfP}{{\widetilde{\mathfrak{P}}}{}}
\nc{\tz}{{\tilde{\zeta}}}
\nc{\tmu}{{\tilde{\mu}}}

\nc{\urho}{\underline{\rho}}
\nc{\uB}{\underline{B}}
\nc{\uC}{{\underline{\mathbb{C}}}}
\nc{\ui}{\underline{i}}
\nc{\uj}{\underline{j}}
\nc{\ofP}{{\overline{\mathfrak{P}}}}
\nc{\oB}{{\overline{\mathcal{B}}}}
\nc{\og}{{\overline{\mathfrak{g}}}}
\nc{\oI}{{\overline{I}}}

\nc{\eps}{\varepsilon}
\nc{\hrho}{{\hat{\rho}}}

\nc{\one}{{\mathbf{1}}}
\nc{\two}{{\mathbf{t}}}

\nc{\Rep}{{\mathop{\operatorname{\rm Rep}}}}
\nc{\Tot}{{\mathop{\operatorname{\rm Tot}}}}
\nc{\Ker}{{\mathop{\operatorname{\rm Ker}}}}
\nc{\Hilb}{{\mathop{\operatorname{\rm Hilb}}}}
\nc{\Ext}{{\mathop{\operatorname{\rm Ext}}}}
\nc{\CHom}{{\mathop{\operatorname{{\mathcal{H}}\it om}}}}
\nc{\GL}{{\mathop{\operatorname{\rm GL}}}}
\nc{\gr}{{\mathop{\operatorname{\rm gr}}}}
\nc{\Id}{{\mathop{\operatorname{\rm Id}}}}
\nc{\de}{{\mathop{\operatorname{\rm def}}}}
\nc{\length}{{\mathop{\operatorname{\rm length}}}}
\nc{\supp}{{\mathop{\operatorname{\rm supp}}}}

\nc{\Cliff}{{\mathsf{Cliff}}}
\nc{\Fl}{\on{Fl}}
\nc{\Fib}{{\mathsf{Fib}}}
\nc{\Coh}{{\on{Coh}}}
\nc{\QCoh}{{\on{QCoh}}}
\nc{\IndCoh}{{\on{IndCoh}}}
\nc{\FCoh}{{\mathsf{FCoh}}}

\nc{\reg}{{\text{\rm reg}}}

\nc{\cplus}{{\mathbf{C}_+}}
\nc{\cminus}{{\mathbf{C}_-}}
\nc{\cthree}{{\mathbf{C}_*}}
\nc{\Qbar}{{\bar{Q}}}
\nc\Eis{\on{Eis}}
\nc\Eisb{\ol\Eis{}}
\nc\Eisr{\on{Eis}^{rat}{}}
\nc\wh{\widehat}
\nc{\Def}{\on{Def_{\check{\fb}}(E)}}
\nc{\barZ}{\overline{Z}{}}
\nc{\barbarZ}{\overline{\barZ}{}}
\nc{\barpi}{\overline\pi}
\nc{\barbarpi}{\overline\barpi}
\nc{\barpip}{\overline\pi{}^+}
\nc{\barpim}{\overline\pi{}^-}

\nc{\fq}{\mathfrak q}

\nc{\fqb}{\ol{\fq}{}}
\nc{\fpb}{\ol{\fp}{}}
\nc{\fpr}{{\fp^{rat}}{}}
\nc{\fqr}{{\fq^{rat}}{}}

\nc{\hattimes}{\wh\otimes}

\nc{\bh}{{\bar{h}}}
\nc{\bOmega}{{\overline{\Omega(\check \fn)}}}

\nc{\seq}[1]{\stackrel{#1}{\sim}}

\nc{\cT}{{\check{T}}}
\nc{\cG}{{\check{G}}}
\nc{\cM}{{\check{M}}}
\nc{\cB}{{\check{B}}}
\nc{\cP}{{\check{P}}}

\nc{\ct}{{\check{\mathfrak t}}}
\nc{\cg}{{\check{\fg}}}
\nc{\cb}{{\check{\fb}}}
\nc{\cn}{{\check{\fn}}}
\nc{\cp}{{\check{\fp}}}
\nc{\cm}{{\check{\fm}}}

\nc{\cLambda}{{\check\Lambda}}

\nc{\cla}{{\check\lambda}}
\nc{\cmu}{{\check\mu}}
\nc{\cnu}{{\check\nu}}
\nc{\ceta}{{\check\eta}}

\nc{\DefbE}{{\on{Def}_{\cB}(E_\cT)}}

\nc{\imathb}{{\ol{\imath}}}
\nc{\rlr}{\overset{\longrightarrow}{\underset{\longrightarrow}\longleftarrow}}

\nc{\oBun}{\overset{\circ}\Bun}
\nc{\BunBbb}{\ol{\ol{Bun}}_B}
\nc{\BunBr}{\Bun_B^{rat}}
\nc{\BunBrsg}{\Bun_B^{rat,\on{s.g.}}}
\nc{\BunBrp}{\Bun_B^{rat,polar}}
\nc{\BunBrpbg}{\Bun_B^{rat,polar,\on{b.g.}}}
\nc{\BunBrpsg}{\Bun_B^{rat,polar,\on{s.g.}}}
\nc{\BunTrp}{\Bun_T^{rat,polar}}
\nc{\BunTrpbg}{\Bun_T^{rat,polar,\on{b.g.}}}
\nc{\BunTrpsg}{\Bun_T^{rat,polar,\on{s.g.}}}
\nc{\BunNr}{\Bun_N^{rat}}
\nc{\BunNre}{\Bun_N^{enh,rat}}
\nc{\BunTr}{\Bun_T^{rat}}
\nc{\Vect}{\on{Vect}}
\nc{\Whit}{\on{Whit}}
\nc{\CT}{\on{CT}}
\nc{\bTb}{\ol{\on{CT}}}

\nc{\bTr}{\on{CT}^{rat}{}}
\nc\jmathr{\jmath^{rat}{}}
\nc{\ux}{\underline{x}}
\nc{\clambda}{{\check\lambda}}
\nc{\calpha}{{\check\alpha}}

\nc{\inftyGrpd}{{\mathsf{Grpd}_\infty}}

\nc{\fset}{\mathsf{fSet}}
\nc{\LocSysG}{\LocSys_{\cG}}
\nc{\Sing}{{\on{Sing}}}
\nc{\dr}{{\on{dR}}}
\nc{\Ind}{\on{Ind}}
\nc{\Sat}{\on{Sat}}
\nc{\Ho}{\on{Ho}}
\nc{\Res}{\on{Res}}
\nc{\sotimes}{\overset{!}\otimes}
\nc{\mmod}{{\on{-}}{\mathbf{mod}}}
\nc{\Maps}{\on{Maps}}
\nc{\CMaps}{{\mathcal Maps}}
\nc{\bMaps}{{\mathbf{Maps}}}

\nc{\dgSch}{\on{DGSch}}
\nc{\dgindSch}{\on{DGindSch}}
\nc{\indSch}{\on{indSch}}
\nc{\Sch}{\mathsf{Sch}}
\nc{\affdgSch}{\on{DGSch}^{\on{aff}}}
\nc{\affSch}{\on{Sch}^{\on{aff}}}

\nc{\Groupoids}{\on{Grpd}}
\nc{\inftypic}{\infty\on{-PicGrpd}}
\nc{\inftyCat}{{\mathsf{Cat}_{\infty}}}
\nc{\MoninftyCat}{\infty\on{-Cat}^{Mon}}
\nc{\SymMoninftyCat}{\infty\on{-Cat}^{\on{SymMon}}}
\nc{\SymMonStinftyCat}{\on{DGCat}^{\on{SymMon}}}
\nc{\MonStinftyCat}{\on{DGCat}^{Mon}}
\nc{\inftystack}{\on{Stk}}
\nc{\inftystackalg}{Stk^{1\text{-}alg}}
\nc{\inftyprestack}{\on{PreStk}}
\nc{\inftydgnearstack}{\on{NearStk}}
\nc{\inftydgstack}{\on{Stk}}
\nc{\inftydgstackalg}{DGStk^{1\text{-}alg}}
\nc{\inftydgprestack}{\on{PreStk}}

\nc{\HC}{\CH\bC}

\nc{\csupp}{\supp}
\nc{\Arth}{\on{Arth}}
\nc{\ArthG}{{\on{Arth}_\cG}}
\nc{\ul}{\underline}
\nc{\rank}{\mathrm{rank}}

\nc{\Z}{\mathcal{Z}}

\nc{\calN}{\N}
\nc{\calW}{\mathcal{W}}
\nc{\calF}{\mathcal{F}}
\nc{\calH}{\mathcal{H}}
\nc{\calO}{\mathcal{O}}
\nc{\calK}{\mathcal{K}}

\nc{\Ran}{\mathsf{Ran}}
\nc{\Jets}{\on{Jets}}
\nc{\act}{\mathsf{act}}
\nc{\Av}{\mathsf{Av}}
\nc{\Ad}{\on{Ad}}
\nc{\BGRan}{BG_{\Ran}}
\nc{\colim}{\on{colim}}
\nc{\codim}{\on{codim}}
\nc{\cpt}{{\on{cpt}}}
\nc{\dR}{{\on{dR}}}
\nc{\DGCat}{\mathsf{DGCat}}
\nc{\DGCatcont}{\on{DGCat}_{cont}}
\nc{\glob}{{\on{glob}}}
\nc{\loc}{{\on{loc}}}

\renewcommand{\op}{{\on{op}}}
\nc{\pt}{{\on{pt}}}
\nc{\PreStk}{{\mathsf{PreStk}}}
\nc{\Cat}{{\mathsf{Cat}}}

\nc{\ShvCat}{{\mathsf{ShvCat}}}

\nc{\restr}[2]{\left. #1 \right |_{#2}}
\nc{\uprestr}[2]{\left. #1 \right |^{#2}}

\nc{\bLoc}{{\mathbf{Loc}}}
\nc{\bGamma}{{\mathbf{\Gamma}}}
\nc{\bLocA}{\mathbf{Loc}^\A}
\nc{\bGammaA}{\mathbf{\Gamma}^\A}
\nc{\bLocB}{\mathbf{Loc}^\B}
\nc{\bGammaB}{\mathbf{\Gamma}^\B}
\nc{\bLocH}{\mathbf{Loc}^\H}
\nc{\bGammaH}{\mathbf{\Gamma}^\H}
\nc{\gen}{\mathsf{gen}}

\nc{\hto}{\hookrightarrow}

\nc{\ext}{\mathsf{ext}}

\nc{\ev}{\mathsf{ev}}
\nc{\ins}{\mathsf{ins}}
\nc{\rat}{\mathsf{rat}}

\nc{\usotimes}[1]{\underset{#1}{\otimes}}
\nc{\ustimes}[1]{\underset{#1}{\times}}
\nc{\uscolim}[1]{\underset{#1}{\colim}}

\nc{\ch}{{\mathfrak{ch}}}

\renc{\fD}{{\Dmod}}
\nc{\fH}{{\mathfrak{H}}}

\nc{\p}{{\mathfrak{p}}}
\renc{\r}{{\mathfrak{r}}}

\nc{\xto}{\xrightarrow}

\renc{\sec}{\section}
\nc{\enh}{\mathsf{enh}}

\renc{\gen}{\mathsf{gen}}
\nc{\BunGBgen}{\Bun_G^{B-\gen}}
\nc{\BunGHgen}{\Bun_G^{H-\gen}}
\nc{\BunGNgen}{\Bun_G^{N-\gen}}

\nc{\Fun}{\mathsf{Fun}}
\nc{\End}{\mathsf{End}}

\nc{\lr}{\xymatrix{ \ar@<-0.4ex>[r] \ar@<.5ex>[l]  & } }
\nc{\rr}{\xymatrix{ \ar@<-0.2ex>[r] \ar@<.7ex>[r]  & } }
\nc{\rrr}{\xymatrix{ \ar@<.0ex>[r] \ar@<.7ex>[r] \ar@<-0.7ex>[r] & } }

\nc{\Stab}{\mathsf{Stab}}
\nc{\Orb}{\mathsf{Orb}}

\renc{\exp}{\mathit{exp}}

\renc{\q}{\mathfrak{q}}

\nc{\virg}[1]{``#1"}

\renc{\bold}[1]{\boldsymbol{#1}}

\nc{\bigt}[1]{\big( #1 \big) }
\nc{\Bigt}[1]{\Big( #1 \Big) }

\nc{\extwhit}{{\CW h}(G,\mathsf{ext})}

\nc{\footcite}{\footnote}

\nc{\GA}{{G(\AA)}}
\nc{\GO}{{G(\OO)}}

\nc{\Shv}{\mathsf{Shv}}
\nc{\inc}{\mathsf{inc}}

\nc{\Par}{\mathsf{Par}}

\renc{\i}{\mathfrak{i}}

\nc{\NA}{N(\AA)}
\nc{\VA}{V(\AA)}

\nc{\Glue}{\mathsf{Glue}}
\nc{\Glued}{\mathsf{Glued}}
\nc{\laxlim}{\text{laxlim}}

\nc{\FT}{\mathsf{FT}}

\nc{\out}{\mathsf{out}}

\nc{\hol}{\mathsf{hol}}
\nc{\Hol}{\on{Hol}}
\nc{\add}{\mathsf{add}}

\nc{\sto}{\rightsquigarrow}
\nc{\squigto}{\rightsquigarrow}

\nc{\fW}{\mathfrak{W}}

\nc{\vrho}{\varrho}

\nc{\counit}{\mathsf{counit}}
\nc{\unit}{\mathsf{unit}}
\nc{\corr}{\mathsf{corr}}
\nc{\Corr}{\mathsf{Corr}}
\nc{\opCorr}{\mathsf{opCorr}}

\nc{\IndSch}{\mathsf{IndSch}}
\nc{\Tate}{{\mathsf{Tate}}}

\nc{\surjto}{\twoheadrightarrow}

\renc{\j}{\mathfrak{j}}

\nc{\J}{\mathcal{J}}

\nc{\pro}{\mathsf{pro}}
\nc{\fty}{\mathsf{ft}}
\nc{\Pro}{\mathsf{Pro}}

\nc{\coact}{\mathsf{coact}}
\nc{\aff}{\mathsf{aff}}

\nc{\Nilp}{\on{Nilp}}

\nc{\Gch}{{\check{G}}}
\nc{\Bch}{{\check{B}}}
\nc{\Tch}{{\check{T}}}
\nc{\Pch}{{\check{P}}}
\nc{\Mch}{{\check{M}}}
\nc{\Qch}{{\check{Q}}}

\nc{\LL}{\mathbb{L}}

\nc{\LS}{{\on{LS}}}

\nc{\x}{\varkappa} 

\nc{\Otimes}{\boldsymbol{\otimes}}
\nc{\Times}{\boldsymbol{\times}}
\nc{\flip}{\text{<}}

\nc{\coeffRan}{\mathsf{coeff}^{\Ran}}

\nc{\Ha}{H(\sA)}
\nc{\Groups}{\mathsf{Groups}}

\nc{\Groth}{\mathsf{Groth}}

\nc{\rlto}{\rightleftarrows}

\nc{\DGCatRan}{\ShvCatCrys(\Ran)}

\nc{\longto}{\longrightarrow}

\renc{\Jets}{\mathsf{Jets}}
\nc{\mer}{\mathsf{mer}}

\nc{\W}{\mathcal{W}}

\nc{\Sect}{\mathsf{Sect}}
\renc{\Maps}{\mathsf{Maps}}

\nc{\y}{\mathtt{y}}
\renc{\x}{\mathtt{x}}

\nc{\un}{{\it un}}
\nc{\indep}{\mathsf{indep}}
\nc{\CoAlg}{\mathsf{CoAlg}}

\nc{\coeff}{\mathsf{coeff}}

\nc{\R}{\mathcal{R}}

\renc{\hat}{\widehat}

\nc{\TK}{T(\mathsf{K})} 
\nc{\TtKK}{\Tt(\mathpzc{K})} 
\nc{\TtK}{\Tt(\mathsf{K})} 
\nc{\KK}{\mathpzc{K}}

\nc{\Dmod}{\mathfrak{D}}

\nc{\curs}[1]{\mathpzc{#1}}

\nc{\Bshv}{\bold{\B}}
\nc{\Bind}{\H_{\indep}}
\nc{\BRan}{\H_{\Ran}}
\nc{\ARan}{\A_{\Ran}}
\nc{\Aind}{\A_{\indep}}

\nc{\GrRan}{\Gr}
\nc{\Gr}{\mathsf{Gr}}
\nc{\GrGRan}{\Gr_{G}}
\nc{\GrGind}{\Gr_{G}^{\indep}}
\nc{\Grind}[1]{\Gr_{#1}^{\indep} }
\nc{\GrGdom}{\curs{Gr}_G}

\nc{\GMapsRan}[1]{\mathsf{GMaps}(X,{#1})}
\nc{\GSectRan}[1]{\mathsf{GSect}({#1}/X)}
\nc{\GMapsind}[1]{\mathsf{GMaps}(X,{#1})^\indep}
\nc{\GSectind}[1]{\mathsf{GSect}({#1}/X)^\indep}
\nc{\GMapsdom}[1]{\curs{GMaps}(X,{#1})}
\nc{\GSectdom}[1]{\curs{GSect}({#1}/X)}

\nc{\chind}{\ch^{\indep}}
\nc{\chdom}{\curs{ch}}

\nc{\QSect}[1]{\curs{QSect}(#1/X)} 
\nc{\QMaps}[1]{\curs{QMaps}(X,#1)} 

\nc{\Zar}{\mathit{Zar}}

\nc{\loccit}{\textit{loc.$\,$cit.}}
\nc{\Crys}{\on{Crys}}
\nc{\ShvCatCrys}{\ShvCat^{\Crys}}

\nc{\BPE}{{\BP E}}
\nc{\BVE}{{\BV E}}
\nc{\BBE}{{\BB E}}

\nc{\Wh}{{{\CW}h}}

\nc{\ChiralCat}{\mathsf{ChiralCat}}
\nc{\RRep}{\mathfrak{R}ep}
\nc{\SSph}{\mathfrak{S}ph}

\nc{\tto}{\twoheadrightarrow}
\nc{\disj}{{\mathsf{disj}}}

\nc{\C}{\CC}

\nc{\good}{\mathsf{good}}
\nc{\triv}{\mathsf{triv}}

\nc{\Alg}{\mathsf{Alg}}
\nc{\CAlg}{\mathsf{CAlg}}
\nc{\Spread}{\mathsf{Spread}}
\nc{\Dom}{\mathsf{Dom}}

\nc{\Jac}{\on{Jac}}
\renc{\CD}[1]{{#1}^{\on{CD}}}

\nc{\String}{\mathsf{String}}

\renc{\min}{{\mathit{min}}}

\nc{\rrep}{\on-\!\mathbf{rep}}

\nc{\WWh}{\mathfrak{W}h}
\nc{\Grpd}{\mathsf{Grpd}}

\nc{\timesdisj}{\overset{\circ}\times}

\renc{\NA}{N(\sA)}
\nc{\chiral}{\mathsf{chiral}}

\nc{\Hopf}{\mathsf{Hopf}}

\nc{\heart}{\heartsuit}
\nc{\kk}{\mathbbm{k}} 

\nc{\HHom}{\CH{om}} 
\nc{\Cone}{\on{Cone}}

\nc{\EE}{\mathbb{E}}
\renc{\HC}{{\on{HC}}}
\nc{\HH}{{\on{HH}}}
\nc{\even}{{\on{even}}}
\nc{\SingSupp}{\on{SingSupp}}
\nc{\Supp}{\on{Supp}}

\nc{\temp}{{\mathsf{temp}}}
\nc{\geom}{{\mathit{geom}}}
\nc{\ren}{{\mathit{ren}}}
\nc{\naive}{{\mathit{naive}}}
\nc{\conaive}{{\mathit{conaive}}}
\nc{\spec}{\mathit{spec}}

\nc{\gch}{\mathfrak{\check{g}}}

\nc{\Hecke}{\on{Hecke}}

\nc{\LSGch}{{\LS_\Gch}}
\nc{\LSBch}{{\LS_\Bch}}
\nc{\LSTch}{{\LS_\Tch}}
\nc{\LSPch}{{\LS_\Pch}}
\nc{\LSMch}{{\LS_\Mch}}
\nc{\LSQch}{{\LS_\Qch}}

\nc{\Hsx}[2]{\H_{{#1} \leftarrow {#2}}}
\nc{\Hdx}[2]{\H_{{#1} \to {#2}}}
\nc{\Hcorr}[3]{ \H_{{#1} \leftarrow {#2} \to {#3}} }
\nc{\Hopcorr}[3]{ \H_{{#1} \to {#2} \leftto {#3}} }

\nc{\ICohsx}[2]{\ICohW_{{#1} \leftarrow {#2}}}
\nc{\ICohdx}[2]{\ICohW_{{#1} \to {#2}}}
\nc{\ICohcorr}[3]{ \ICohW_{{#1} \leftarrow {#2} \to {#3}} }
\nc{\ICohopcorr}[3]{ \ICohW_{{#1} \to {#2} \leftto {#3}} }

\nc{\QCohsx}[2]{\QCohW_{{#1} \leftarrow {#2}}}
\nc{\QCohdx}[2]{\QCohW_{{#1} \to {#2}}}
\nc{\QCohcorr}[3]{ \QCohW_{{#1} \leftarrow {#2} \to {#3}} }
\nc{\QCohopcorr}[3]{ \QCohW_{{#1} \to {#2} \leftto {#3}} }

\renc{\AA}{\bbA}
\nc{\Asx}[2]{\AA_{{#1} \leftarrow {#2}}}
\nc{\Adx}[2]{\AA_{{#1} \to {#2}}}
\nc{\Acorr}[3]{ \AA_{{#1} \leftarrow {#2} \to {#3}} }
\nc{\Aopcorr}[3]{ \AA_{{#1} \to {#2} \leftto {#3}} }

\nc{\Bsx}[2]{\B_{{#1} \leftarrow {#2}}}
\nc{\Bdx}[2]{\B_{{#1} \to {#2}}}
\nc{\Bcorr}[3]{ \B_{{#1} \leftarrow {#2} \to {#3}} }
\nc{\Bopcorr}[3]{ \B_{{#1} \to {#2} \leftto {#3}} }

\nc{\Ssx}[2]{\S_{{#1} \leftarrow {#2}}}
\nc{\Sdx}[2]{\S_{{#1} \to {#2}}}
\nc{\Scorr}[3]{ \S_{{#1} \leftarrow {#2} \to {#3}} }
\nc{\Sopcorr}[3]{ \S_{{#1} \to {#2} \leftto {#3}} }

\nc{\ICohzero}[3]{\ICoh_0 \bigt{#1 \times_{{#2}_\dR} #3}}
\nc{\IndCohzero}{\ICohzero}

\nc{\form}[3]{#1 \times_{{#2}_\dR} #3 }

\nc{\ind}{{\mathsf{ind}}}
\nc{\oblv}{{\mathsf{oblv}}}

\nc{\Aff}{\mathsf{Aff}}
\nc{\dgAff}{\Aff}

\nc{\coev}{\mathsf{coev}}

\nc{\bE}{\mathbf{E}}

\nc{\ShvCatH}{{\ShvCat^{\bbH}}}
\nc{\ShvCatQW}{\ShvCat^{\QCohW}}

\nc{\bbimod}{\on{-}\mathbf{bimod}}

\nc{\Tw}{\mathsf{Tw}}
\nc{\TwTr}{\mathsf{TwTr}}
\nc{\Arr}{\mathsf{Arr}}

\nc{\bDelta}{\bold\Delta}
\nc{\BiCat}{\mathsf{BiCat}}
\nc{\Seg}{\mathsf{Seg}}
\nc{\Cart}{\mathsf{Cart}}
\nc{\Bimod}{\mathsf{Bimod}}
\nc{\lax}{\mathit{lax}}

\nc{\pr}{\mathsf{pr}}

\nc{\zero}{ \{ 0 \}   }
\nc{\Perf}{\mathsf{Perf}}

\nc{\leftto}{\leftarrow}
\nc{\lto}{\leftto}
\nc{\xlto}[1]{\xleftarrow{#1}}

\nc{\ltemp}{{}^\temp}
\nc{\TwCorr}{\mathsf{TwCorr}}

\nc{\Affover}[1]{{\Aff_{/#1}}}
\nc{\Affoverop}[1]{{( \Affover{#1})^\op}}

\nc{\AffOver}[2]{{(\Aff_{#1})_{/#2}}}

\nc{\AffOverop}[2]{{( \AffOver{#1}{#2})^\op}}

\nc{\aft}{{\mathit{aft}}}

\renc{\vert}{{\mathit{vert}}}
\nc{\horiz}{{\mathit{horiz}}}
\nc{\type}{{\mathit{type}}}
\nc{\adm}{{\mathit{adm}}}

\nc{\g}{\mathfrak{g}}

\nc{\free}{\mathsf{free}}

\nc{\Sform}{{S \times_{S_\dR} S}}
\nc{\Yform}{{\Y \times_{\Y_\dR} \Y}}
\nc{\SdR}{ {S_{\dR}}}

\nc{\laft}{{\mathit{laft}}}

\nc{\Affevcocaft}{\Aff_{\aft}^{< \infty}}
\nc{\Affaftevcoc}{\Aff_{\aft}^{< \infty}}

\nc{\Affevcoclfp}{\Aff_{\lfp}^{< \infty}}

\nc{\Schevcoclfp }{\Sch_{\lfp}^{< \infty}}
\nc{\Schevcocaft}{\Sch_{\aft}^{< \infty}}
\nc{\Schaftevcoc}{\Sch_{\aft}^{< \infty}}

\nc{\Stkevcoc}{\Stk^{< \infty}}
\nc{\Stkevcoclfp}{\Stk_{\lfp}^{< \infty}}
\nc{\Stkperfevcoclfp}{\Stk_{\mathit{perf},\lfp}^{< \infty}}
\nc{\Stkperflfp}{\Stk_{\mathit{perf},\lfp}}
\nc{\Stklfp}{\Stk_{\lfp}}

\nc{\evcoc}{\mathit{e.c.}}

\nc{\ICoh}{\IndCoh}

\nc{\citep}{\cite}

\renc{\H}{\bbH}

\nc{\uno}{\mathbbm{1}}

\nc{\CohBig}{{\Coh^{-\infty}}}

\nc{\Tang}{\mathbb{T}}

\nc{\LieAlg}{\mathsf{LieAlg}}

\nc{\Serre}{{\on{Serre}}}

\nc{\MPreStk}{\mathsf{MPreStk}}

\nc{\all}{{\on{all}}}

\nc{\QCohwedge}{\bbQ^\wedge}
\nc{\ICohwedge}{\bbI^\wedge}
\nc{\ICohW}{\ICohwedge}
\nc{\QCohW}{\QCohwedge}

\nc{\ShvCatA}{\ShvCat^{\AA}}
\nc{\ShvCatB}{{\ShvCat^\B}}

\nc{\naiveto}{{\xto{\naive}}}
\nc{\conaiveto}{{\xto{\conaive}}}

\nc{\strong}{\mathsf{strong}}
\nc{\costrong}{\mathit{costrong}}

\nc{\conv}{\mathit{conv}}

\nc{\Q}{\bbQ}

\nc{\bY}{\mathbf{Y}}
\nc{\Loop}{\mathsf{LOOP}}

\nc{\DG}{{\on{DG}}}

\nc{\coind}{\mathsf{coind}}

\nc{\co}{\on{co}}

\nc{\laftdef}{{\mathit{laft-def}}}

\nc{\qsmooth}{{\mathit{qs.smooth}}}
\nc{\smooth}{{\mathit{smooth}}}

\nc{\LKE}{\on{LKE}}
\nc{\RKE}{\on{RKE}}

\nc{\ShvCatAco}{\ShvCatA_{\co}}
\nc{\ShvCatHco}{\ShvCatH_{\co}}

\nc{\Stk}{\mathsf{Stk}}

\renc{\2}{(\infty,2)}
\renc{\1}{(\infty,1)}

\nc{\doubleCat}{\mathsf{doubleCat}}
\nc{\Spaces}{\mathcal{S}\!\mathit{paces}}

\nc{\ALG}{\mathsf{ALG}}
\nc{\MAPS}{\mathsf{MAPS}}

\nc{\CAT}{\mathsf{CAT}}

\nc{\oneCat}{{\Cat_{\1}}}
\nc{\oneCAT}{{\CAT_{\1}}}

\nc{\twoCat}{{\Cat_{\2}}}
\nc{\twoCAT}{{\CAT_{\2}}}

\nc{\DGCAT}{\mathsf{DGCAT}}
\nc{\twoCatDG}{{\CAT_{\2}^\DG}}
\nc{\twoCATDG}{{\CAT_{\2}^\DG}}

\nc{\twoCATDGw}{{\CAT_{\2, w*}^\DG}}
\nc{\twoCATDGww}{{\CAT_{\2, ww*}^\DG}}

\nc{\AlgBimod}{\Alg^{\mathit{bimod}}}
\nc{\AlgBimodDGCat}{\AlgBimod(\DGCat)}
\nc{\ALGBimod}{\ALG^{\mathit{bimod}}}
\nc{\twoAlgBimod}{\ALGBimod}

\nc{\rev}{{\on{rev}}}

\nc{\lfp}{{\mathit{lfp}}}

\nc{\RBeck}{{\on{R-BC}}}
\nc{\LBeck}{{\on{L-BC}}}

\nc{\schem}{\mathit{schem}}
\nc{\proper}{\mathit{proper}}

\nc{\res}{{\mathit{res}}}

\nc{\UQCoh}{\U^{\QCoh}}
\nc{\UQ}{\UQCoh}

\nc{\LieAlgbd}{{\on{Lie-algbd}}}

\nc{\LY}{{L\Y}}
\nc{\LZ}{{L\Z}}

\nc{\TangQ}{\Tang^{\QCoh}}

\nc{\Fil}{{\on{Fil}}}
\nc{\AssGr}{\on{assoc-gr}}

\nc{\red}{{\mathit{red}}}

\nc{\Cech}{\on{Cech}}

\nc{\FormMod}{\mathsf{FormMod}}
\nc{\FormModunderStkevcoc} {\FormMod_{\Stk^{< \infty}/}^\lfp }

\nc{\vDmod}{\virg{\Dmod}}
\nc{\shDmod}{\Dmod^\to}

\nc{\HeckeEigen}{\on{Hecke-Eigen}}
\nc{\HE}{\HeckeEigen}
\nc{\bigHE}{{\HE}^\ren}
\nc{\HEAut}{\HE^{\Aut}}
\nc{\bigHEAut}{\HE^{\Aut, \ren}}
\nc{\Spr}{\on{Spr}}

\nc{\LSG}{{\LS_G}}
\nc{\LSP}{{\LS_P}}
\nc{\LSQ}{{\LS_Q}}
\nc{\LSB}{{\LS_B}}
\nc{\LST}{{\LS_T}}
\nc{\LSM}{{\LS_M}}

\nc{\LSGtriv}{{\LS_G^\triv}}
\nc{\LSPtriv}{{\LS_P^\triv}}
\nc{\LSQtriv}{{\LS_Q^\triv}}

\nc{\PP}{{\mathbb{P}}}

\nc{\Gm}{{\GG_m}}

\nc{\sigmatriv}{{\sigma_\triv}}

\nc{\shift}{{\Leftarrow}}
\nc{\unshift}{{\Rightarrow}}

\nc{\PQ}{{Q \subseteq P}}
\nc{\TwPQ}{[\PQ]}
\nc{\twPQ}{\TwPQ}
\nc{\St}{\on{St}}

\nc{\Weyl}{\on{Weyl}}
\nc{\Dol}{{\on{Dol}}}
\nc{\coSing}{\on{coSing}}

\nc{\DN}{\mathfrak{DN}}
\nc{\NTwGlued}{\N_{\Glued} }
\nc{\NGlued}{\N_{\Glued} }
\nc{\Der}{\on{Der}}
\nc{\Prim}{\on{Prim}}

\nc{\alt}{{\mathit{alt}}}
\nc{\Aalt}{A^\alt}
\nc{\cl}{{cl}}

\nc{\unip}{{\mathsf{unip}}}

\nc{\xleftto}{\xleftarrow}
\nc{\xlefto}{\xleftarrow}

\nc{\Nch}{{\check{\N}}}

\nc{\nilp}{{\mathit{nilp}}}

\nc{\QSmooth}{\mathsf{QSmooth}}

\nc{\Nglob}{{\N_{\glob}}}

\nc{\acts}{\mathrel{\reflectbox{$\righttoleftarrow$}}}

\nc{\inv}{\mathit{inv}}

\nc{\DSingY}{{\Dmod(\Sing(Y))^\unshift}}

\nc{\sigmaA}{{\, \sigma,A}}
\nc{\mixed}{\mathsf{mixed}}

\nc{\Grid}{\on{Grid}}
\nc{\grid}{\mathit{grid}}

\EpigaVolumeYear{4}{2020} \EpigaArticleNr{9} \ReceivedOn{December 5,
2019}
\InFinalFormOn{April 7, 2020}
\AcceptedOn{May 25, 2020}

\title{The spectral gluing theorem revisited}
\titlemark{The spectral gluing theorem revisited}

\author{Dario Beraldo}
\address{Institut de Math\'ematiques de Toulouse; UMR 5219, Universit\'e de Toulouse; CNRS, UPS, 118 route de Narbonne, F--31062 Toulouse Cedex 9, France}
\email{darioberaldo@gmail.com}

\authormark{D. Beraldo}

\AbstractInEnglish{We strengthen the \emph{gluing theorem} (the main result of \cite{AG2}) occurring on the spectral side of the geometric Langlands conjecture. The gluing theorem of \cite{AG2} provides a fully faithful embedding of $\ICoh_\N(\LSG)$ into $\Glue_P I(G,P)$, a DG category glued out of \virg{Fourier coefficients} parametrized by the standard parabolic subgroups of $G$. Our refinement explicitly identifies the essential image of such embedding. The key idea is to introduce, for each inclusion $\PQ$, a \emph{mixed} Fourier coefficient $I(G, \PQ)$. Formation of the latter is covariant in $P$ and contravariant in $Q$, hence the assignment $[\PQ] \squigto I(G, \PQ)$ is a functor out of \emph{twisted arrows} of standard parabolics. Our theorem states that $\ICoh_\N(\LSG)$ is the limit of such functor.}



\TitleInFrench{Le th\'eor\`eme de recollement spectral revisit\'e}

\AbstractInFrench{Nous renfor\c{c}ons le \emph{th\'eor\`eme de recollement} (le r\'esultat principal de \cite{AG2}) qui intervient du c\^ot\'e spectral de la conjecture de Langlands g\'eom\'etrique. Le th\'eor\`eme de recollement de \cite{AG2} donne un plongement pleinement fid\`ele de $\ICoh_\N(\LSG)$ dans $\Glue_P I(G,P)$, une DG cat\'egorie recoll\'ee \`a partir de \virg{coefficients de Fourier} param\'etr\'es par les sous-groupes paraboliques standards de $G$. Notre raffinement d\'ecrit explicitement l'image essentielle de ce plongement. L'id\'ee-cl\'e est d'introduire, pour chaque inclusion $\PQ$, un coefficient de Fourier \emph{mixte} $I(G, \PQ)$. La formation de ce dernier est covariante en $P$ et contravariante en $Q$, donc la fonction $[\PQ] \squigto I(G, \PQ)$ est un foncteur \`a partir de \emph{fl\`eches tordues} de paraboliques standards. Notre th\'eor\`eme \'etablit que $\ICoh_\N(\LSG)$ est la limite de ce foncteur.}

\acknowledgement{Research partially supported by EPSRC programme grant EP/M024830/1 Symmetries and Correspondences and by ERC-2016-ADG-741501.}

\begin{document}


\removeabove{0.5cm}
\removebetween{0.5cm}
\removebelow{0.5cm}

\maketitle

\begin{prelims}

\DisplayAbstractInEnglish

\bigskip

\DisplayKeyWords

\medskip

\DisplayMSCclass

\bigskip

\languagesection{Fran\c{c}ais}

\bigskip

\DisplayTitleInFrench

\medskip

\DisplayAbstractInFrench

\end{prelims}


\newpage

\setcounter{tocdepth}{2}

\tableofcontents


\sec{Introduction}

Let $X$ be a smooth connected complete curve over a ground field $\kk$, algebraically closed and of characteristic zero. Let $G$ be a connected reductive group over $\kk$ and $\LSG$ the derived stack of de Rham $G$-local systems on $X$. We choose a Borel subgroup  $B \subseteq G$, regarded as fixed throughout.

\sssec{}

The DG category $\ICoh_\N(\LSG)$ of \emph{ind-coherent sheaves on $\LSG$ with nilpotent singular support} (see \cite{ICoh} and \cite{AG1}) is one of the two protagonists of the geometric Langlands program.
In \cite{AG1}, it is shown that $\ICoh_\N(\LSG)$ is bigger than the more familiar $\QCoh(\LSG)$, in the precise sense that there is a colocalization\footnote{A technical term that means: an adjunction with fully faithful left adjoint.}
\begin{equation} 	\nonumber
\begin{tikzpicture}[scale=1.5]
\node (a) at (0,1) {$\QCoh(\LSG)$};
\node (b) at (2.5,1) {$\ICoh_\N(\LSG)$.};
\path[right hook ->,font=\scriptsize,>=angle 90]
([yshift= 1.5pt]a.east) edge node[above] {$\Xi$ } ([yshift= 1.5pt]b.west);
\path[->>,font=\scriptsize,>=angle 90]
([yshift= -1.5pt]b.west) edge node[below] {$\Psi$ } ([yshift= -1.5pt]a.east);
\end{tikzpicture}
\end{equation}
It is also explained that the difference between these two DG categories is a manifestation of the non-commutativity of $G$, accounted for by the existence of proper parabolic subgroups. For instance, the two DG categories coincide iff $G$ is a torus, in which case there are no proper parabolic subgroups.

\sssec{} \label{sssec: idea of AG gluing}

This idea takes a more precise form in the spectral gluing theorem of \cite{AG2}, which amounts to:
\begin{itemize}
\item
a glued category 
$$
\Glue := \underset{P \in \Par^\op}{\laxlim} \; I(G,P)
$$ 
consisting of $2^{\rank(G)}$ pieces, one for each standard\footnote{\virg{Standard} means \virg{containing the chosen Borel $B$}.} parabolic subgroup;
\item
a fully faithful functor 
$$
\gamma: \ICoh_\N(\LSG) \hto \Glue.
$$
\end{itemize}
Here $\Par$ denotes the poset of standard parabolics, including $G$, with respect to inclusions.

\sssec{}

In analogy with the extended Whittaker category construction (see \cite{Outline}, \cite{ext-whit}), the gluing components $I(G,P)$ should be regarded as categories of Fourier modes and $\gamma$ should be regarded as a Fourier decomposition. For instance, the category of $G$-Fourier modes $I(G,G)$ coincides with  $\QCoh(\LSG)$. 
%

The present work originates from the observation that such Fourier decomposition is imperfect, in that $\gamma$ is not essentially surjective. 
%
%
The goal of this paper is to correct this imperfection by identifying exactly which collections of Fourier coefficients belong to the essential image of $\gamma$.

\ssec{The case of $G= \mathrm{GL}_2$}

Let us explain the construction of \cite{AG2} and our improvement in the simplest case of $G$ of semisimple rank $1$. For definiteness, we set $G=\mathrm{GL}_2$ and keep this assumption in place until Section \ref{ssec:intro-general G}.

\sssec{}

Since there is only one proper standard parabolic $B$, the difference between $\QCoh(\LSG)$ and $\ICoh_\N(\LSG)$ is controlled by $I(G,B)$, which is a certain DG category of sheaves on the formal completion $(\LSG)^\wedge_{\LSB}$ of the map $\LSB \to \LSG$ (the map that forgets the horizontal flag).
In more detail, $I(G,B)$ is a hybrid DG category
$$
I(G,B) := \ICoh_0((\LSG)^\wedge_{\LSB})
$$
that sits between $\QCoh((\LSG)^\wedge_{\LSB})$ and $\ICoh((\LSG)^\wedge_{\LSB})$.

\sssec{} \label{defn:Icohzero}

The definition of $\ICoh_0$ goes as follows. For a map $\Y \to \Z$ of quasi-smooth stacks, set:
$$
\ICoh_0(\Z^\wedge_\Y)
:=
\ICoh(\Z^\wedge_\Y)
\ustimes{\ICoh(\Y)}
\QCoh(\Y),
$$
where the map $\QCoh(\Y) \to \ICoh(\Y)$ is the natural inclusion and $\ICoh(\Z^\wedge_\Y) \to \ICoh(\Y)$ is the ind-coherent pullback along $\Y \to \Z^\wedge_\Y$.
In \cite{centerH}, \cite{shvcatHH} and \cite{chiral-hom}, we gave explanations for the \emph{raison d'\^{e}tre} of $\ICoh_0$.
By construction, $\ICoh_0(\Z^\wedge_\Y)$ belongs to a natural sequence of colocalizations:
\begin{equation} 	\nonumber
\begin{tikzpicture}[scale=1.5]
\node (a) at (-0.20,1) {$\QCoh(\Z^\wedge_\Y)$};
\node (b) at (2,1) {$\ICoh_0(\Z^\wedge_\Y)$};
\node (c) at (4,1) {$\ICoh(\Z^\wedge_\Y)$.};
\path[right hook ->,font=\scriptsize,>=angle 90]
([yshift= 1.5pt]a.east) edge node[above] { $\Xi_{\Z^\wedge_\Y}$} ([yshift= 1.5pt]b.west);
\path[->>,font=\scriptsize,>=angle 90]
([yshift= -1.5pt]b.west) edge node[below] { $\Psi_{\Z^\wedge_\Y}$ } ([yshift= -1.5pt]a.east);
\path[right hook ->,font=\scriptsize,>=angle 90]
([yshift= 1.5pt]b.east) edge node[above] {$  $} ([yshift= 1.5pt]c.west);
\path[->>,font=\scriptsize,>=angle 90]
([yshift= -1.5pt]c.west) edge node[below] {$ $} ([yshift= -1.5pt]b.east);
\end{tikzpicture}
\end{equation}

\sssec{}

In our particular case of $I(G,B) := \ICoh_0((\LSG)^\wedge_{\LSB})$, let us consider the adjunction
\begin{equation} 	\nonumber
\begin{tikzpicture}[scale=1.5]
\node (a) at (0,1) {$\QCoh((\LSG)^\wedge_{\LSB})$};
\node (b) at (3,1) {$\ICoh_0((\LSG)^\wedge_{\LSB})$};
\path[right hook ->,font=\scriptsize,>=angle 90]
([yshift= 1.5pt]a.east) edge node[above] { $\Xi_{G,B}$} ([yshift= 1.5pt]b.west);
\path[->>,font=\scriptsize,>=angle 90]
([yshift= -1.5pt]b.west) edge node[below] { $\Psi_{G,B}$ } ([yshift= -1.5pt]a.east);
\end{tikzpicture}
\end{equation}
and the functor
\begin{equation} \label{eqn:gluing from G to B intro}
\QCoh(\LSG) 
\xto{(\wh \p_B)^*}
\QCoh((\LSG)^\wedge_{\LSB})
\xto{\Xi_{G,B}}
\ICoh_0((\LSG)^\wedge_{\LSB}),
\end{equation}
where $\wh \p_B: (\LSG)^\wedge_{\LSB} \to \LSG$ is the obvious map.
Whenever we have a functor $F: \C \to \D$ between two DG categories, we can form the glued DG category 
$$
\Glue(\C \to \D),
$$
whose objects are triples $ \bigt{ c \in \C, d \in \D; \eta: F(c) \to d }$, where $\eta$ is an arrow (not necessarily an isomorphism!) in $\D$.

\sssec{}

For $G=\mathrm{GL}_2$, the \emph{spectral gluing theorem} of \cite{AG2} is an explicit fully faithful functor
\begin{equation} \label{eqn: weak spectral gluing GL2}
\gamma: 
\ICoh_\N(\LSG)
\hto
\Glue := \Glue 
\Bigt{ \QCoh(\LSG)  \to \ICoh_0((\LSG)^\wedge_{\LSB}) }.
\end{equation}
As mentioned, the starting point of the present work is the observation that such inclusion \emph{is not an equivalence}. More precisely, while $\QCoh(\LSG)$ is too small to match $\ICoh_\N(\LSG)$, the glued DG category $\Glue$ is too big.

\begin{example}
The triple 
$$
(\O_{\LSG}, 0; \eta = 0) 
\in
\Glue(\QCoh(\LSG)  \to \ICoh_0((\LSG)^\wedge_{\LSB}))
$$
does not belong to the essential image of $\gamma$. More generally, a triple of the form $(\F, 0;0)$ belongs to the essential image of $\gamma$ if and only if $\F$ is pushed forward from the open locus of \emph{irreducible} $G$-local systems.
\end{example}

\sssec{}

In this paper, we state and prove a \emph{strong spectral gluing theorem} which explicitly identifies the essential image of the inclusion $\gamma$. For $G=\mathrm{GL}_2$, our result can be stated now: informally, rather than gluing along (\ref{eqn:gluing from G to B intro}), we glue along the cospan(=op-correspondence) below.

\begin{thm} [Strong spectral gluing for $G=\mathrm{GL}_2$] \label{thm:intro-GL2 case}
For $G=\mathrm{GL}_2$, the DG category $\ICoh_\N(\LSG)$ is naturally equivalent to the limit of the diagram
\begin{equation} 
\nonumber
\begin{tikzpicture}[scale=1.5]
\node (00) at (0,0) {$\QCoh(\LSG)$};
\node (10) at (3,0) {$\QCoh((\LSG)^\wedge_{\LSB})$.};
\node (11) at (3,1) {$\ICoh_0((\LSG)^\wedge_{\LSB})$};
\path[->,font=\scriptsize,>=angle 90]
(00.east) edge node[above] {$(\wh \p_B)^*$} (10.west); 
\path[->>,font=\scriptsize,>=angle 90]
(11.south) edge node[right] {$\Psi_{G,B}$} (10.north);
\end{tikzpicture}
\end{equation} 

\end{thm}

\begin{rem}
Let us explicitly describe the target DG category in the two versions of the spectral gluing theorem (still in the case of $\mathrm{GL}_2$).
The target of the version in \cite{AG2} is the DG category $\Glue$, whose objects are triples 
$$
\Bigt{
\F \in \QCoh(\LSG), 
\G \in I(G,B),
\eta: \Xi_{G,B} ((\wh \p_B)^*(\F)) \to \G
},
$$
where the arrow $\eta$ is not required to be an isomorphism.
The target of our version is the limit of the op-correspondence above, which is the full subcategory of $\Glue$ spanned by objects satisfying the following additional requirement: the morphism 
$\wt\eta: (\wh \p_B)^*(\F) \to \Psi_{G,B} (\G)$ in $\QCoh((\LSG)^\wedge_{\LSB})$, obtained from $\eta$ by adjunction, must be an isomorphism.
\end{rem}

\ssec{A baby version}

The point of the above theorem is that $\gamma: \ICoh_\N(\LSG) \hto \Glue$ factors as the composition
$$
\ICoh_\N(\LSG)
\xto{ \;\;\simeq \;\;}
\QCoh(\LSG)
\ustimes{\QCoh((\LSG)^\wedge_{\LSB})}
\ICoh_0((\LSG)^\wedge_{\LSB})
\hto
\Glue.
$$
As we are about to illustrate, a similar phenomenon occurs in a much simpler (yet entertaining) situation. 
The following discussion might appear unrelated to the strong spectral gluing for $\mathrm{GL}_2$; we will explain later that there is a tight relation between the two.

\sssec{} \label{sssec: blow up situation}
\nc{\Bl}{\mathsf{Bl}}

Consider the cartesian square
\begin{equation} 
\nonumber
\begin{tikzpicture}[scale=1.5]
\node (00) at (0,0) {$\pt = \Spec(\kk)$};
\node (10) at (2,0) {$V$,};
\node (01) at (0,1) {$E$};
\node (11) at (2,1) {$\Bl$};
\path[right hook ->,font=\scriptsize,>=angle 90]
(00.east) edge node[above] {$i_0$} (10.west); 
\path[right hook ->,font=\scriptsize,>=angle 90]
(01.east) edge node[above] {$i$} (11.west); 
\path[->>,font=\scriptsize,>=angle 90]
(01.south) edge node[right] {$p_E$} (00.north);
\path[->>,font=\scriptsize,>=angle 90]
(11.south) edge node[right] {$\pi$} (10.north);
\end{tikzpicture}
\end{equation} 
where $V \simeq \AA^n$ is a finite dimensional vector scheme, $\Bl$ its blow-up at the origin and $E$ the exceptional divisor.
The baby situation we will describe amounts to expressing $\Dmod(V)$ as a DG category glued out $\Dmod(\Bl)$ and $\Dmod(\pt) \simeq \Vect_\kk$.

\sssec{}

Let us look at the following table of analogies, where the baby version is on the left and the adult version on the right:
\begin{equation} 	\nonumber
\begin{tikzpicture}[scale=1.5]
\node (A) at (-2,2) {$\Dmod(\pt)$};
\node (B) at (0,2) {$\Dmod(V)$;};
\node (a) at (2.5,2) {$\QCoh(\LSG)$};
\node (b) at (5.5,2) {$\ICoh_\N(\LSG)$};
\path[right hook ->,font=\scriptsize,>=angle 90]
([yshift= 1.5pt]a.east) edge node[above] {$\Xi$ } ([yshift= 1.5pt]b.west);
\path[->>,font=\scriptsize,>=angle 90]
([yshift= -1.5pt]b.west) edge node[below] {$\Psi$ } ([yshift= -1.5pt]a.east);
\path[right hook ->,font=\scriptsize,>=angle 90]
([yshift= 1.5pt]A.east) edge node[above] {$(i_0)_*$ } ([yshift= 1.5pt]B.west);
\path[->>,font=\scriptsize,>=angle 90]
([yshift= -1.5pt]B.west) edge node[below] {$(i_0)^!$ } ([yshift= -1.5pt]A.east);
\node (s) at (2.5,1) {$\QCoh(\LSG)$};
\node (a) at (2.5,0) {$\QCoh((\LSG)^\wedge_{\LSB})$};
\node (b) at (6,0) {$\ICoh_0((\LSG)^\wedge_{\LSB})$.};
\path[->,font=\scriptsize,>=angle 90]
([yshift= 1.5pt]s.south) edge node[right] { $(\wh \p_B)^*$} ([yshift= 1.5pt]a.north);
\path[right hook ->,font=\scriptsize,>=angle 90]
([yshift= 1.5pt]a.east) edge node[above] { $\Xi_{G,B}$} ([yshift= 1.5pt]b.west);
\path[->>,font=\scriptsize,>=angle 90]
([yshift= -1.5pt]b.west) edge node[below] { $\Psi_{G,B}$ } ([yshift= -1.5pt]a.east);
\node (S) at (-2,1) {$\Dmod(\pt)$};
\node (A) at (-2,0) {$\Dmod(E)$};
\node (B) at (0,0) {$\Dmod(\Bl)$;};
\path[->,font=\scriptsize,>=angle 90]
([yshift= 1.5pt]S.south) edge node[right] { $(p_E)^!$} ([yshift= 1.5pt]A.north);
\path[right hook ->,font=\scriptsize,>=angle 90]
([yshift= 1.5pt]A.east) edge node[above] { $i_*$} ([yshift= 1.5pt]B.west);
\path[->>,font=\scriptsize,>=angle 90]
([yshift= -1.5pt]B.west) edge node[below] { $i^!$ } ([yshift= -1.5pt]A.east);
\end{tikzpicture}
\end{equation}

\sssec{}

Consider now the functor
$$
\gamma_{\mathit{baby}}: 
\Dmod(V)
\longto 
\Glue
\bigt{
\Dmod(\pt) 
\xto{i_* \circ (p_E)^!} 
\Dmod(\Bl)
},
\hspace{.4cm}
\F \squigto
\bigt{
(i_0)^!(\F), \pi^!(\F), \eta
},
$$
where the arrow $\eta: i_* p_E^! (i_0)^!(\F) \longto \pi^!(\F)$ in $\Dmod(\Bl)$ is the obvious one obtained by adjunction.
One can prove directly that $\gamma_{\mathit{baby}}$ is fully faithful (but not an equivalence). This is the baby version of the original spectral gluing theorem.

\sssec{}

On the other hand, let us look at the pushout prestack $\pt \sqcup_E \Bl$ and at its DG category of $\fD$-modules:
$$
\Dmod(\pt \sqcup_E \Bl)
\simeq
\Dmod(\pt)
 \ustimes{\Dmod(E)} 
 \Dmod(\Bl).
$$
There is an evident inclusion
\begin{equation} \label{eqn:inclusion B to AG}
\Dmod(\pt)
\ustimes{\Dmod(E)} 
\Dmod(\Bl)
\hto
\Glue(\Dmod(\pt) \to \Dmod(\Bl)).
\end{equation}
The baby version of the strong spectral gluing theorem states that $\gamma_{\mathit{baby}}$ factors as the composition of the equivalence
$$
(i_0^!, \pi^!):\Dmod(V) 
\xto{\; \; \simeq \; \; }
\Dmod(\pt) \ustimes{\Dmod(E)} \Dmod(\Bl)
$$
followed by the inclusion (\ref{eqn:inclusion B to AG}).

\sssec{} \label{sssec:explain-baby-version}

The baby version of the main theorem is more than just an analogy, as we now explain.
Two properties of $\LSG$ (as well as all the $\LSP$'s) are at play here. 

\medskip

The most important one is that these stacks are quasi-smooth. 
Any quasi-smooth stack $\Y$ has an associated stack of singularities $\Sing(\Y)$, which is the reduced classical truncation of the $(-1)$-shifted cotangent bundle of $\Y$. This stack is used to define the notion of singular support for (ind-)coherent sheaves, see \cite{BIK} and \cite{AG1}. 
In the case at hand, $\Sing(\LSG) := \Arth_G$ is\footnote{after choosing an Ad-invariant bilinear form on $\fg$} the space of \emph{geometric Arthur parameters} consisting of pairs $(\sigma, A)$, where $\sigma$ is a $G$-local system and $A$ a horizontal section of the flat vector bundle $\fg_{\sigma}$.

\medskip

The second property is that $\LS_G$ is a global complete intersection stack. By choosing an atlas appropriately, our statements about $\LSG$ reduce to statements about a global complete intersection scheme, that is, a scheme written as a fiber product $U \times_V \pt$ with $U$ smooth affine and $V$ a vector group.

\sssec{}

Now let $Y =U \times_V \pt$ be a global complete intersection scheme. 
Then $\Sing(Y) := H^{-1}(T^* Y)$ is the kernel of the codifferential of the map $U \to V$, that is, the fiber product $(Y \times V^*) \times_{T^* U} U$.
We regard $\Dmod(\Sing(Y))$ as a symmetric monoidal DG category (equipped with its pointwise tensor product) and show that $\ICoh(Y)$ carries an action of a certain modification of $\Dmod(\Sing(Y))$. Such modification, denoted $\DSingY$, is the symmetric monoidal DG category obtained from $\Dmod(\Sing(Y))$ via a \virg{shift of grading} that uses the $\Gm$-action on the fibers of the projection $\Sing(Y) \to Y$.

\sssec{} \label{sssec:pippo}

Roughly speaking, the passage from the baby theorem (which is a gluing theorem for categories of $\fD$-modules) to the adult version (which concerns categories of ind-coherent sheaves) is the process of tensoring up with $\ICoh(Y)$ over $\DSingY$. We will be more precise on this later, in Section \ref{sec:microlocal reformulation}.

This idea is borrowed from \cite{AG2}, with the following difference: rather than constructing an action of $\DSingY$ on $\ICoh(Y)$, the authors of \cite{AG2} construct an action of $\Dmod(\bbP\Sing(Y))$ on the quotient DG category $\ICoh(Y)/\QCoh(Y)$.
The latter has the advantage of being defined for any quasi-smooth scheme $Y$, rather than just for global complete intersections.

\ssec{The case of general $G$} \label{ssec:intro-general G}

Let us now describe the shape of the strong spectral gluing theorem in the case $G$ has higher semisimple rank.
To formulate the statement, it will be necessary to borrow some notions from the theory of $\H$, developed in \cite{centerH} and \cite{shvcatHH}. The main points of this theory are reviewed in Section \ref{sssec:Review H}.

\sssec{}

We start by recalling that the spectral gluing theorem of \cite{AG2} amounts to an explicit fully faithful functor
$$
\gamma:
\ICoh_\N(\LSG)
\hto
\Glue 
:=
\underset{P \in \Par^\op}\laxlim \; 
\ICoh_0((\LSG)^\wedge_{\LSP}),
$$
where $\Par$ is the poset of all standard parabolics subgroups of $G$.

\begin{rem}
To recover the statement for $G=\mathrm{GL}_2$, just note that $\ICoh_0((\LSG)^\wedge_{\LSG}) \simeq \QCoh(\LSG)$.
\end{rem}

\sssec{} \label{sssec:description of objects of GLue}

An object of $\Glue$ consists of:
\begin{itemize}
\item
for each $P \in \Par$, an object $\F_P \in \ICoh_0((\LSG)^\wedge_{\LSP})$;

\smallskip 

\item
for each inclusion $\PQ$, an arrow 
$$
\eta_{\PQ}:(\wh\p_{G,\PQ})^{!,0}(\F_P) \longto \F_Q
$$
in $\ICoh_0((\LSG)^\wedge_{\LSQ})$, where
$$
(\wh\p_{G,\PQ})^{!,0}:
\ICoh_0((\LSG)^\wedge_{\LSP})
\longto
\ICoh_0((\LSG)^\wedge_{\LSQ})
$$
is the obvious pullback functor;

\smallskip

\item
some natural compatibilities that we do not spell out here (we will spell them out in the main body of the paper, see also \cite[Sections 4.1-4.2]{AG2}).

\end{itemize}

\sssec{} \label{sssec:Review H}

Our improved version of the theorem identifies the essential image of $\gamma$. To define the relevant full subcategory of $\Glue$, we first need to review the main features of the $\H$-construction.

\begin{itemize}
\item

The starting point is the definition of the DG category $\ICoh_0(\Z^\wedge_\Y)$ attached to a map $\Y \to \Z$ of quasi-smooth stacks. See Section \ref{defn:Icohzero}, as well as \cite[Section 3.2]{AG2} and \cite[Section 3]{centerH}.

\smallskip 

\item
The $\ICoh_0$ construction, applied to the diagonal $\Y \to \Y \times \Y$, yields the DG category 
$$
\H(\Y) := \ICoh_0((\Y \times \Y)^\wedge_\Y).
$$
The convolution monoidal structure on $\ICoh((\Y \times \Y)^\wedge_\Y)$ restricts to a monoidal structure on $\H(\Y)$.

\smallskip 

\item
More generally, for a map $\Y \leftto \X \to \Z$ of quasi-smooth stacks, we consider
$$
\Hcorr \Y \X \Z := \ICoh_0((\Y \times \Z)^\wedge_\X),
$$
which is naturally an $(\H(\Y), \H(\Z))$-bimodule. We also use the notations $\Hdx \X \Z := \Hcorr \X \X \Z$ and $\Hsx \Y \X := \Hcorr \Y \X \X$.
For instance, we have:
\begin{eqnarray}
\nonumber
& & \Hcorr  \Y \X \pt \simeq \ICoh_0(\Y^\wedge_\X) \\
\nonumber
& & \Hcorr \X  \X \X \simeq \H(\X) \\
\nonumber
& & \Hdx \X \pt \simeq \QCoh(\X) \\
\nonumber
& & \Hcorr \pt \X \pt \simeq \Dmod(\X).
\end{eqnarray}

\smallskip 

\item
For a diagram $\V \leftto \W \leftto \X \to \Y \to \Z$ of quasi-smooth stacks, there is a natural equivalence
$$
\Hsx \V \W
\usotimes{\H(\W)}
\Hcorr \W \X \Y 
\usotimes{\H(\Y)}
\Hdx \Y \Z
\xto{\;\; \simeq \;\;}
\Hcorr \V \X \Z
$$
of $(\H(\V), \H(\Z))$-bimodule DG categories.

\smallskip 

\item
In the above setting, the $(\H(\W), \H(\Y))$-bimodule $\Hcorr \W \X \Y $ is left and right dualizable, with both duals canonically identified with $\Hcorr \Y \X \W $. It follows from this that $\QCoh(\Y) \simeq \Hdx \Y \pt$ is a bimodule for the left action of $\H(\Y)$ and the right action of $\Hcorr \pt \Y \pt \simeq \Dmod(\Y)$.

\smallskip 

\item
The diagonal $\Delta_\Y$ induces a monoidal pushforward functor $\QCoh(\Y) \to \H(\Y)$ that admits a continuous and conservative right adjoint. This implies that $\H(\Y)$ is rigid. 

\smallskip 

\item
For $\X \to \Y$ a map of quasi-smooth stacks and $\N$ a closed conical subset of $\Sing(\Y)$, the monoidal DG category $\H(\Y)$ acts on $\ICoh(\Y)$, $\ICoh_\N(\Y)$ and $\QCoh(\Y^\wedge_\X)$.

\end{itemize}

\sssec{}

Now, for $\PQ$, there is an $\H(\LSP)$-linear diagram
\begin{equation} 	\nonumber
\begin{tikzpicture}[scale=1.5]
\node (a) at (0,1) {$\QCoh(\LSP)$};
\node (b) at (2,1) {$\QCoh((\LSP)^\wedge_\LSQ)$};
\node (c) at (5,1) {$\ICoh_0((\LSP)^\wedge_\LSQ).$};
\path[->,font=\scriptsize,>=angle 90]
([yshift= 1.5pt]a.east) edge node[above] {$\wh\p^*_\PQ$} ([yshift= 1.5pt]b.west);
\path[right hook ->,font=\scriptsize,>=angle 90]
([yshift= 1.5pt]b.east) edge node[above] {$ \Xi_{\PQ}$} ([yshift= 1.5pt]c.west);
\path[->>,font=\scriptsize,>=angle 90]
([yshift= -1.5pt]c.west) edge node[below] {$\Psi_{\PQ}$} ([yshift= -1.5pt]b.east);
\end{tikzpicture}
\end{equation}
Let us now push forward along $\p_P: \LSP \to \LSG$ in the $\H$-sense, that is, we tensor up with $\Hsx \LSG \LSP$ over $\H(\LSP)$.
Thanks to the third and fourth item above, we obtain an $\H(\LSG)$-linear diagram 
\begin{equation} 	\nonumber
\begin{tikzpicture}[scale=1.5]
\node (a) at (0,1) {$\ICoh_0((\LSG)^\wedge_\LSP)$};
\node (b) at (4.2,1) {$\Hsx \LSG \LSP
\usotimes{\H(\LSP)}
\QCoh((\LSP)^\wedge_\LSQ)$};
\node (c) at (8.2,1) {$\ICoh_0((\LSG)^\wedge_\LSQ).$};
\path[->,font=\scriptsize,>=angle 90]
([yshift= 1.5pt]a.east) edge node[above] {$(\wh\p_{G, \PQ})^{!,P-\temp}$} ([yshift= 1.5pt]b.west);
\path[right hook ->,font=\scriptsize,>=angle 90]
([yshift= 1.5pt]b.east) edge node[above] {$ \Xi_{G,\PQ}$} ([yshift= 1.5pt]c.west);
\path[->>,font=\scriptsize,>=angle 90]
([yshift= -1.5pt]c.west) edge node[below] {$\Psi_{G,\PQ}$} ([yshift= -1.5pt]b.east);
\end{tikzpicture}
\end{equation}

\sssec{}

The DG category $\Glue$ was constructed by disregarding $\Psi_{G, \PQ}$: indeed, the composition of the two arrows from left to right is exactly the functor $(\wh\p_{G,\PQ})^{!,0}$ that appeared in Section \ref{sssec:description of objects of GLue}.

\medskip

The key proposal of the present paper is that one should instead disregard $\Xi_{G, \PQ}$ and \emph{keep} $\Psi_{G, \PQ}$. Indeed, our theorem states that an object of $\Glue$ as above belongs to the essential image of $\gamma$ if and only if the following condition is satisfied: 
for any $\PQ$, the arrow $\eta_{\PQ}$ induces an isomorphism
$$
(\wh\p_{G,\PQ})^{!,P-\temp}(\F_P) \longto \Psi_{G,\PQ}(\F_Q).
$$

\sssec{}

To state this more formally, we set
$$
\ICoh_0((\LSG)^\wedge_{\LSQ})^{P-\temp}
:=
\Hsx \LSG \LSP
\usotimes{\H(\LSP)}
\QCoh((\LSP)^\wedge_\LSQ),
$$
and consider the DG category
$$
\lim_{[\PQ] \in \Tw(\Par)^\op}
\,
\ICoh_0((\LSG)^\wedge_{\LSQ})^{P-\temp},
$$
where the limit is taken along the cospans
$$
\ICoh_0((\LSG)^\wedge_{\LSP})
\xto{\; 
(\wh \p_{G, \PQ})^{!, P-\temp}
\; 
}
\ICoh_0((\LSG)^\wedge_{\LSQ})^{P-\temp}
\xleftarrow{\; 
\Psi_{G, \PQ}
\;
}
\ICoh_0((\LSG)^\wedge_{\LSQ})
$$
displayed above. The indexing $1$-category $\Tw(\Par)$ is the poset of \emph{twisted arrows} of $\Par$, see Section~\ref{sssec:twisted arrows}. The decoration \virg{temp} stands for \virg{tempered}, see \cite[Example 1.3.4]{shvcatHH} for an explanation of the terminology.

\sssec{}

We can now formulate an early version of the main result of this paper (the official version appears later as Theorem \ref{mainthm:strong-ICOHN}).

\begin{thm}[Strong spectral gluing] \label{mainthm-intro}
There is a natural $\H(\LSG)$-linear equivalence
$$
\gamma^{strong}:
\ICoh_\N(\LSG)
\xto{\;\;\simeq\;\;}
\lim_{[\PQ] \in \Tw(\Par)^\op} 
\,
\ICoh_0((\LSG)^\wedge_{\LSQ})^{P-\temp}.
$$
\end{thm}

\sssec{}

In Example \ref{example trivial loc sys}, we will show that, by restricting to the trivial $G$-local system, the above theorem is related to the following local (that is, independent of the curve $X$) analogue.

\begin{thm} \label{thm:D-mod on Nilp local}
Let $\N_\g \subseteq \g$ be the nilpotent cone of the Lie algebra of $G$.
For each pair of standard parabolics $\PQ$, consider the natural correspondence
$$
G \times^P \u_P
\twoheadleftarrow
G \times^Q \u_P
\hto
G \times^Q \u_Q
$$
over $\N_\g$. 
Then the pullback functors yield an equivalence
\begin{equation} \label{eqn:theorem for nilcone}
\Dmod(\N_\g)
\simeq
\lim_{[\PQ ] \in \Tw(\Par^\op)^\op} \Dmod(G \times^Q \u_P).
\end{equation}
\end{thm}

\begin{rem}
For $G=\mathrm{GL}_2$, Theorem \ref{thm:D-mod on Nilp local} states that $\Dmod(\N_{\mathfrak{gl}_2})$ is equivalent to $\Dmod(\pt) \times_{\Dmod(\PP^1)} \Dmod(T^*(\PP^1))$, the DG category obtained from the diagram $\pt \leftto \PP^1 \hto T^*\PP^1$ by pullback. This situation is very similar to the blow-up situation of Section \ref{sssec: blow up situation} and it will be proven in the same way.
\end{rem}

\ssec{Automorphic gluing, a preview}

The geometric Langlands conjecture calls for an equivalence between $\ICoh_\N(\LSG)$ and the DG category $\Dmod(\Bun_\Gch)$ of $\fD$-modules on the stack $\Bun_\Gch(X)$ of $\Gch$-bundles on the same curve $X$. Here $\Gch$ is the Langlands dual group of $G$.
Under geometric Langlands, the spectral gluing theorem ought to correspond to a gluing statement for $\Dmod(\Bun_\Gch)$. We conclude the introduction with an informal\footnote{We plan to give a rigorous account of this material elsewhere.} discussion of such a statement, which we call \emph{automorphic gluing}.

\sssec{}

To reduce cuttler, let us swap $G$ with $\Gch$ and formulate a gluing conjecture for $\Dmod(\Bun_G)$.
A key ingredient is the \emph{tempered subcategory}, denoted by $\ltemp\Dmod(\Bun_G)$. The definition appears in \cite[Section 12]{AG1} and an equivalent characterization is given in \cite{omega-antitemp}; see also \cite[Sections 1.3-1.4]{shvcatHH}.
In general, we can define the tempered subcategory of any DG category $\C$ equipped with an action of the spherical monoidal DG category $\Sph_G$.
As with $\QCoh$ and $\ICoh$, there is a colocalization (the right adjoint is called the \virg{temperization functor}):
\begin{equation}  \label{eqn:temp fucntor}
\begin{tikzpicture}[scale=1.5]
\node (a) at (0,1) {$\ltemp\C$};
\node (b) at (1.5,1) {$\C$.};
\path[right hook ->,font=\scriptsize,>=angle 90]
([yshift= 1.5pt]a.east) edge node[above] {$ $ } ([yshift= 1.5pt]b.west);
\path[->>,font=\scriptsize,>=angle 90]
([yshift= -1.5pt]b.west) edge node[below] {$\temp$ } ([yshift= -1.5pt]a.east);
\end{tikzpicture}
\end{equation}
We denote by $\C^\circ$ the right orthogonal of $\ltemp\C$ inside $\C$: this might be called the \emph{anti-tempered subcategory}.

\sssec{}

For any $P \in \Par$, we also need to introduce a DG category $I(G,P)^{\mathrm{aut}}$ as follows. (The superscript \virg{$\mathrm{aut}$} stands for \virg{automorphic}.)
Consider the prestack $\Bun_G^{P-\gen}$ of $G$-bundles on $X$ equipped with a generic reduction to $P$, see \cite{Barlev} for the precise definition. 
There are canonical maps 
$$
\Bun_M \xleftarrow {\fq_P}  \Bun_P \xto{f_P} \Bun_G^{P - \gen},
$$
where $M$ is the Levi quotient of $P$. We set\footnote{This definition is taken from \cite{Outline}, but we warn the reader that the notation is different.}
$$
I(G,P)^{\mathrm{aut}}
:=
\Dmod(\Bun_G^{P-\gen}) 
\ustimes{\Dmod(\Bun_P)}
\ltemp\Dmod(\Bun_M),
$$
where the left arrow is the pullback $f_P^!$ and the right arrow is the (fully faithful) composition
$$
\ltemp\Dmod(\Bun_M) \hto \Dmod(\Bun_M) \xto{\q_P^!} \Dmod(\Bun_P).
$$
As in \cite[Section 6]{Outline}, there is an action of $\Sph_G$ on $I(G,P)$: this allows us to later take the tempered and anti-tempered subcategories of $I(G,P)$.

\sssec{}

Now, the automorphic analogue of the spectral gluing of \cite{AG2} states that the DG categories $I(G,P)^{\mathrm{aut}}$ assemble into a glued DG category 
$$
\Glue^{\mathrm{aut}}
:=
\underset{P \in \Par^\op}\laxlim \; 
I(G,P)^{\mathrm{aut}}
$$
and that there is a fully faithful functor
$\gamma^{\mathrm{aut}}:
\Dmod(\Bun_G) \to \Glue^{\mathrm{aut}}$.
The strong automorphic gluing conjecture amounts to an explicit description of the essential image of $\gamma^{\mathrm{aut}}$ as a limit over $\Tw(\Par)$, similarly to what is done in the present paper for the spectral side. Below, we give the details for $G$ of semisimple rank $1$.

\sssec{}

From now until the end of the introduction, let $G$ be of semisimple rank one.
In this case, $\Glue^{\mathrm{aut}}$ has two terms: $\ltemp\Dmod(\Bun_G)$ and 
$$
I(G,B)^{\mathrm{aut}}
\simeq
\Dmod(\Bun_G^{B-\gen}) 
\ustimes{\Dmod(\Bun_B)}
\Dmod(\Bun_T).
$$
Since $T := \on{Levi}(B)$ is abelian, the inclusion $\ltemp\Dmod(\Bun_T) \hto \Dmod(\Bun_T)$ is an equivalence.
The inclusion $I(G,B) \hto \Dmod(\Bun_G^{B-\gen})$ admits a continuous right adjoint; this allows to define the \emph{enhanced constant term} functor
$$
\CT_B^{\enh}:
\Dmod(\Bun_G) \xto{f_B^!}
\Dmod(\Bun_G^{B-\gen})
\tto
I(G,B).
$$
This functor is $\Sph_G$-linear and it is equipped with a (necessarily $\Sph_G$-linear) left adjoint, $\Eis_B^{\enh}$. In particular, $\CT_B^{\enh}$ preserves the (anti-)tempered subcategories.

\sssec{}

The weaker version of the automorphism gluing theorem states that the natural functor
\begin{equation} \label{eqn:weak aut gluing GL2}
\Dmod(\Bun_G)
\longto
\Glue \Bigt{ 
\ltemp\Dmod(\Bun_G)
\xto{\CT_B^{\enh}}
I(G,B)^{\mathrm{aut}}
}
\end{equation}
is fully faithful. One can check that such fully faithfulness boils down to the fully faithfulness of the functor
$$
\Dmod(\Bun_G)^\circ
\xto{\; \; \CT_B^\enh \; \;}
\Bigt{  I(G,B)^{\mathrm{aut}} }^\circ.
$$

\begin{example}
As evidence for the validity of the latter statement, consider the example of $\omega_{\Bun_G}$. This object has been proven to be anti-tempered (for any nonabelian $G$, not just in semisimple rank $1$) in \cite{omega-antitemp}. Now, the fact that the counit of the adjunction
$$
\Eis_B^\enh \CT_B^\enh (\omega_{\Bun_G}) \longto \omega_{\Bun_G}
$$
is an isomorphism is a quick consequence of the contractibility of the space of rational maps from $X$ to $G/B$, proven in \cite{contract} and \cite{Barlev}.
This example is very close to being a proof: indeed, for $G$ of rank one, we expect that $\Dmod(\Bun_G)^\circ$ is generated under colimits by $\omega_{\Bun_G}$.
\end{example}

\sssec{}

Now note that the functor appearing on the RHS of \eqref{eqn:weak aut gluing GL2} factors as
$$
\ltemp\Dmod(\Bun_G)
\to
\ltemp I(G,B)^{\mathrm{aut}}
\hto 
I(G,B)^{\mathrm{aut}}.
$$
Moreover, the inclusion $\ltemp I(G,B)^{\mathrm{aut}} \hto I(G,B)^{\mathrm{aut}}$ admits a right adjoint, see \eqref{eqn:temp fucntor}.
The strong automorphic gluing in semisimple rank $1$ reads as follows:

\begin{conj}
For $G$ of semisimple rank $1$, the commutative diagram
\begin{equation} 
\nonumber
\begin{tikzpicture}[scale=1.5]
\node (00) at (0,0) {$\ltemp\Dmod(\Bun_G)$};
\node (10) at (3.5,0) {$\ltemp I(G,B)^{\mathrm{aut}}$.};
\node (01) at (0,1) {$\Dmod(\Bun_G)$};
\node (11) at (3.5,1) {$I(G,B)^{\mathrm{aut}}$};
\path[->,font=\scriptsize,>=angle 90]
(00.east) edge node[above] {$\CT_B^\enh$} (10.west); 
\path[ ->,font=\scriptsize,>=angle 90]
(01.east) edge node[above] {$\CT_B^\enh$} (11.west); 
\path[->>,font=\scriptsize,>=angle 90]
(01.south) edge node[right] {$\temp$} (00.north);
\path[->>,font=\scriptsize,>=angle 90]
(11.south) edge node[right] {$\temp$} (10.north);
\end{tikzpicture}
\end{equation} 
is a fiber square.
\end{conj}

\sssec{}

This conjecture boils down to proving that the functor
$$
\Dmod(\Bun_G)^\circ
\xto{\; \; \CT_B^\enh \; \;}
\Bigt{  I(G,B)^{\mathrm{aut}} }^\circ
$$
is an equivalence. As above, we expect that $\Bigt{  I(G,B)^{\mathrm{aut}} }^\circ$ is generated under colimits by $\omega_{\Bun_G^{B-\gen}}$. This would make the equivalence manifest.

\ssec{Organization of the paper}

In Section \ref{sec:global complete intersect}, we construct the action of $\DSingY$ on $\ICoh(Y)$ and discuss some of its properties.
Section \ref{sec:strong spectral gluing} is devoted to the definition of the term appearing in the main Theorem \ref{mainthm-intro}: our glued DG category and of the gluing functor that ought to realize the equivalence.
Section \ref{sec:microlocal reformulation} explains how to reduce the proof of the main theorem to a simpler statement about categories of $\fD$-modules on schemes: this is made possible by the results of Section \ref{sec:global complete intersect}.
Finally, Section \ref{sec:proof} extends the combinatorics of \cite{AG2} to prove the statements left open in Section \ref{sec:microlocal reformulation}.

\ssec{Acknowledgements}

I would like to thank Dima Arinkin and Dennis Gaitsgory for their patient explanations and their insight: I owe them a great deal.
I am obliged to Ian Grojnowski and Sam Raskin for several useful conversations, as well as to the anonymous referees for their comments and corrections.

\sec{Global complete intersections} \label{sec:global complete intersect}

After a preliminary section on the \emph{shift of grading trick}, we discuss the relationship between ind-coherent sheaves and spaces of singularities in the case of global complete intersections.

\ssec{Shift of grading}

Consider the $\infty$-category 
$$
\Gm \rrep^{\mathsf{weak}} 
:=
(\QCoh(\Gm), \star) \mmod,
$$
see \cite{AG1} or \cite{Be}.
We will make extensive usage of the \virg{shift of grading} automorphism of $\Gm \rrep^{\mathsf{weak}}$, defined in \cite{AG1} and denoted by $\C \squigto \C^\shift$. We denote by $\C \squigto \C^\unshift$ its inverse.

\sssec{}

The definition of $\C^\shift$ goes as follows:
\begin{itemize}
\item
take the invariant category $\C^{\Gm}$, equipped with its action of $\Rep(\Gm)$;

\smallskip

\item
twist this action of $\Rep(\Gm)$ by the symmetric monoidal autoequivalence $\Rep(\Gm) \to \Rep(\Gm)$ induced by the assignment
$$
\{V_{i,n}\} 
\squigto
\{V_{i-2n,n}\},
$$
where $V_{i,n}$ denotes the component in cohomological degree $i$ and weight $n$ of a graded DG vector space $V$;

\smallskip

\item
finally, set $\C^\shift := \C^{\Gm} \otimes_{\Rep(\Gm)} \Vect$, where $\Gm$ acts trivially on $\Vect$.
\end{itemize}

\begin{example}
If $\Gm$ acts trivially on $\C$, then $\C^\shift \simeq \C$.
\end{example}

\begin{example} \label{example:Gm-action on algebras}

Let $A$ be a graded DG algebra.  Then $\C = A \mod$ is naturally an object of $\Gm \rrep^{\mathsf{weak}}$. As above, we write $A =\{A_{i,n}\}$ where $A_{i,n}$ is the component in cohomological degree $i$ and grading degree $n$.
We have:
$$
(A \mod)^\shift
= 
A^{\shift} \mod,
$$
where $A^{\shift}$ is the graded DG algebra with components $(A^{\shift})_{i,n} = A_{i+2n,n}$.
\end{example}

\sssec{}

The latter example shows that the shift/unshift automorphisms do not commute with the forgetful functor $\Gm \rrep \to \DGCat$. On the other hand, by construction, $(\C^{\shift})^\Gm$ and $\C^\Gm$ are equivalent as DG categories.

\sssec{}

Recall that $\Gm \rrep^{\mathsf{weak}}$ is symmetric monoidal (compatibly with the forgetful functor to $\DGCat$). The shift automorphism preserves the symmetric monoidal structure. Hence, it also preserves relative tensor products within $\Gm \rrep^{\mathsf{weak}} $.

\begin{example} \label{example:Koszul-duality}
Let $V$ be a finite dimensional vector space and consider the dilation $\Gm$-action on $V^*$, which corresponds to the tautological grading on $\Sym V$.
This induces $\Gm$-actions on $\QCoh(V^*)$, on $\Omega V := \pt \times_V \pt$ and on $\ICoh(\pt \times_V \pt)$. Under the Koszul duality equivalence
$$
\ICoh(\Omega V)
\simeq
(\Sym V[-2]) \mod,
$$
elements of $\Sym^n V$ are placed in bidegree $(2n,n)$. Hence, $\ICoh(\Omega V) \simeq \QCoh(V^*)^\unshift$. This equivalence swaps the convolution and the pointwise monoidal structures. By the theory of singular support, we deduce that $\QCoh(\Omega V) \simeq \QCoh((V^*)^\wedge_0)^\unshift$.
\end{example}

\begin{example}
In the situation of the previous example, there is also a $\Gm$-action on $\Dmod(V^*)$. We obtain an equivalence 
$$
\Dmod(V^*)^\unshift \simeq W_2(V^*) \mod,
$$
where $W_2(V^*)$ is the $2$-shifted Weyl algebra whose $\partial$ variables have bidegree $(-2, -1)$.
 The Fourier-Deligne transform yields an equivalence $\Dmod(V^*)^\unshift \simeq \Dmod(V)^\shift$.

\end{example}

\begin{prop}
Consider again the scheme $Y= \Omega V$ and recall the action of $\QCoh(Y)$ on $\ICoh(Y)$, as well as the action of $\QCoh(Y)$ on $\Vect$ given by pullback along the inclusion $\pt \hto \Omega V$. There is a canonical equivalence
$$
\ICoh(Y)
\usotimes{\QCoh(Y)}
\Vect 
\simeq 
\Dmod(V^*)^\unshift.
$$
\end{prop}

\begin{proof}
We use the above Koszul duality equivalence, together with the fact that the shift automorphism is symmetric monoidal:  
\begin{eqnarray}
\ICoh(Y)
\usotimes{\QCoh(Y)}
\Vect 
& \simeq & 
\QCoh(V^*)^\unshift
\usotimes{\QCoh((V^*)^\wedge_0)^\unshift}
\Vect 
\\
\nonumber
& \simeq & 
\Bigt{ 
\QCoh(V^*)
\usotimes{\QCoh((V^*)^\wedge_0)}
\Vect
}
^\unshift
\nonumber
\\
& \simeq & 
\QCoh
\bigt{ 
V^*/{(V^*)^\wedge_0}
}
^\unshift
\nonumber
\\
& \simeq & 
\Dmod(V^*)^\unshift, 
\nonumber
\end{eqnarray}
where the third equivalence follows (for example) by formal smoothness of $(V^*)^\wedge_0$, together with proper descent for $\ICoh$.
\end{proof}

\begin{rem}
In fact, the above equivalence is the simplest nontrivial instance of the equivalence
\begin{equation}
\ICoh(\Y)
\usotimes{\QCoh(\Y)}
\QCoh(y)
\simeq
\Dmod
\bigt{  \!
\restr{\Sing(\Y)}y
\!
}
^\unshift,
\end{equation}
valid for any quasi-smooth stack and any $\kk$-point $y \in \Y$. This will be addressed in another paper.
\end{rem}

\ssec{Singular support for global complete intersections} \label{ssec:global complete intersect}

We show that for $Y$ a global complete intersection (see the definition below), the DG category $\ICoh(Y)$ admits an action of $\DSingY$, which encodes the notion of singular support for coherent sheaves.

\sssec{}

Recall from \cite[Section 2.3]{AG1} that, for a quasi-smooth scheme $Y$, the scheme of singularities $\Sing(Y)$ is a classical scheme living over $Y^\cl$, the classical truncation of $Y$. It is defined as the relative spectrum of the $\O_{Y^\cl}$-algebra $\Sym_{\O_{Y^\cl}} H^1(\Tang_Y)$.
We always equip $\Sing(Y)$ with the $\Gm$-action coming from the obvious grading of the above symmetric algebra.
Such $\Gm$-action induces a strong (and in particular a weak) $\Gm$-action on $\Dmod(\Sing(Y))$: this structure allows us to consider the shifted DG category $\DSingY$.

\sssec{}

We say that a DG scheme $Y$  is a \emph{global complete intersection}\footnote{We warn the reader that this definition of global complete intersection is not standard: for us, the presentation as a fiber product is part of the data.} if it is presented as a fiber product $Y = U \times_V \pt$, with $U$ smooth affine and $V\simeq \AA^n$ a vector space.
A global complete intersection is obviously quasi-smooth and conversely any quasi-smooth scheme is Zariski locally of this form, see e.g. \cite[Corollary 2.1.6]{AG1}.

\sssec{} \label{sssec:singularities and ICOhzero}

Let $Y = U \times_V \pt$ be a global complete intersection. Such $Y$, together with its presentation, is regarded as fixed throughout the remainder of Section \ref{sec:global complete intersect}. Observe that
$$
\Sing(Y) \simeq (Y \times V^*) \times_{T^* U} U
$$
as classical schemes, where the left map in the fiber product is the dual of the differential and the right one is the zero section.
In particular, we have $\Gm$-equivariant closed embeddings
$$
\Sing(Y) \hto Y \times V^* \hto U \times V^*,
$$
where the $\Gm$-action on $V^*$ and on $T^*U$ is dilation (along the fibers of $T^*U \to U$ in the second case).

\sssec{}

Consider the pullback action of $\QCoh(U) \simeq \ICoh(U)$ on $\ICoh(Y)$, as well as the convolution action of $\ICoh(\pt \times_V \pt)$ on $\ICoh(Y)$. A simple diagram chase shows that these two actions commute.
It follows that
$$
(\QCoh(U), \otimes) \otimes (\ICoh(\pt \times_V \pt), \star)
$$
acts on $\ICoh(Y)$.
The Koszul duality equivalence 
$$
\bigt{ \ICoh(\pt \times_V \pt), \star }
\simeq 
\bigt{
\QCoh(V^*)^\unshift, \otimes
}
$$
transforms this into an action of
$$
\bigt{
\QCoh(U \times V^*)^\unshift, \otimes
}
$$
on $\ICoh(Y)$. From a slightly different point of view, this action is discussed in \cite[Section 5.4.6]{AG1}.

\sssec{}

By \cite[Corollary 5.4.7]{AG1}, the above action of $\QCoh(U \times V^*)^\unshift$ factors through the monoidal colocalization   
$$
\QCoh(U \times V^*)^\unshift
\tto
\QCoh
\bigt{
(U \times V^*)^\wedge_{\Sing(Y)}
}^\unshift.
$$
Then the monoidal functor
$$
\Dmod(\Sing(Y))^\unshift
\longto
\QCoh
\bigt{
(U \times V^*)^\wedge_{\Sing(Y)}
}^\unshift,
$$
induced by pullback along 
$$
(U \times V^*)^\wedge_{\Sing(Y)} \to \Sing(Y)_\dR,
$$
yields an action of $(\DSingY, \otimes)$ on $\ICoh(Y)$: this is the action we were looking for.

\begin{cor} \label{cor:singsupp in terms of Dmod unshift for global comple intersection}
For $Y$ as above and $N$ a closed conical subset of $\Sing(Y)$, the following full subcategories of $\ICoh(Y)$ are equivalent:
$$
\ICoh_{N}(Y)
\simeq
\Dmod(N)^{\unshift}
\usotimes
{\Dmod(\Sing(Y))^\unshift}
\ICoh(Y).
$$
\end{cor}

\begin{proof}
Tautological from \cite[Corollary 5.4.7]{AG1}.
\end{proof}

\begin{example} \label{example:QCoh = zero section}
In particular:
$$
\QCoh(Y)
\simeq
\Dmod(Y)
\usotimes
{\Dmod(\Sing(Y))^\unshift}
\ICoh(Y).
$$
\end{example}

\begin{example} \label{example:relation with AG2}

Let $\Sing(Y)^\circ := \Sing(Y) - O_Y$ be the complement of the zero section and $\ICoh(Y)^\circ$ the quotient $\ICoh(Y)/\QCoh(Y)$. We have:
\begin{equation} \label{eqn:relation with AG2}
\ICoh(Y)^\circ
\simeq
\Dmod(\Sing(Y)^\circ)^\unshift
\usotimes{\DSingY}
\ICoh(Y).
\end{equation}
The symmetric monoidal forgetful functor 
$$
\Dmod(\PP\Sing(Y))
\simeq
\Dmod(\Sing(Y)^\circ)^{\Gm, \mathit{strong}}
\to
\Dmod(\Sing(Y)^\circ)^{\Gm}
\simeq
\bigt{
\Dmod(\Sing(Y)^\circ)^\unshift
}^{\Gm}
\to
\Dmod(\Sing(Y)^\circ)^\unshift
$$ 
yields an action of $\Dmod(\PP\Sing(Y))$ on $\ICoh(Y)^\circ$, which is the one constructed in \cite{AG2}.

\end{example}

\ssec{The singular codifferential} \label{ssec:singular codiff}

In this section, we use the $\DSingY$-action on $\ICoh(Y)$ constructed above to study the DG category $\ICoh_0(Y^\wedge_X) \simeq \Hcorr Y X \pt$.

\sssec{}

Let $Y =  U \times_V \pt$ be a global complete intersection as before, $X$ a quasi-smooth scheme, and $f: X \to Y$ an arbitrary map.
We consider the standard correspondence
$$
\Sing(X)
\xleftarrow{\fs_f}
X  \times_Y \Sing(Y)
\xto{\ft_f}
\Sing(Y),
$$
where the left map is called \emph{singular codifferential}, see \cite[Section 2.4]{AG1}.

\sssec{}

The exterior tensor product yields a natural equivalence
$$
\ICoh(Y^\wedge_X)
\xleftarrow{\simeq}
\ICoh(Y)
\usotimes{\Dmod(Y)}
\Dmod(X).
$$ 
To see this, it suffices to combine \cite[Proposition 3.1.2]{AG2} with the $1$-affineness of $Y_\dR$ and \cite[Proposition 3.1.9]{ShvCat}.
In our case, since $\Dmod(Y)$ acts on $\ICoh(Y)$ via $\Dmod(Y) \to \DSingY$ the monoidal pullback functor, we obtain that
\begin{equation} \label{eqn:ICoh-formal via Sing}
\ICoh(Y^\wedge_X)
\simeq
\ICoh(Y)
\usotimes{\DSingY}
\Dmod(X \times_Y \Sing(Y))^\unshift.
\end{equation}
The next result shows that $\ICoh_0(Y^\wedge_X)$ can be expressed is a similar way.

\begin{prop} \label{prop:ICohzero via DSing}
Let $X$ be a quasi-smooth scheme equipped with a map $f: X \to Y$. 
Under the equivalence \eqref{eqn:ICoh-formal via Sing}, the subcategory 
$$
\ICoh_0(Y^\wedge_X) \subseteq \ICoh(Y^\wedge_X)
$$ 
identifies with the subcategory
$$
\ICoh(Y)
\usotimes{\Dmod(\Sing(Y))^\unshift}
\Dmod( \fs_f^{-1}(O_X))^\unshift
\subseteq
\ICoh(Y)
\usotimes{\Dmod(\Sing(Y))^\unshift}
\Dmod(X \times_Y \Sing(Y))^\unshift
$$
induced by the inclusion $\fs_f^{-1}(O_X) \subseteq X \times_Y \Sing(Y)$.
\end{prop}

\begin{proof}
Setting
\begin{eqnarray}
\nonumber
& & \ICoh(Y)^\circ  := \ICoh(Y)/\QCoh(Y)
\\
 & & 
 \nonumber
 \ICoh(Y^\wedge_X)^\circ := \ICoh(Y^\wedge_X)/\QCoh(Y^\wedge_X) \\
  & & 
 \nonumber
 \ICoh_0(Y^\wedge_X)^\circ := \ICoh_0(Y^\wedge_X)/\QCoh(Y^\wedge_X),
\end{eqnarray}
it is easy to see that the natural functor
$$
 \ICoh_0(Y^\wedge_X)
 \longto
  \ICoh(Y^\wedge_X)
  \ustimes{ \ICoh(Y^\wedge_X)^\circ}
   \ICoh_0(Y^\wedge_X)^\circ
$$
is an equivalence. Now, the Arinkin-Gaitsgory action of $\Dmod(\PP\Sing(Y))$ on $\ICoh(Y)^\circ$, see Example \ref{example:relation with AG2}, yields
$$
 \ICoh(Y^\wedge_X)^\circ \simeq \ICoh(Y)^\circ 
 \usotimes{\Dmod(\bbP \Sing(Y))}
 \Dmod(X \times_Y \bbP \Sing(Y) ).
$$
By \cite[Section 3.2.10]{AG2}, we also have
$$
 \ICoh_0(Y^\wedge_X)^\circ \simeq \ICoh(Y)^\circ 
 \usotimes{\Dmod(\bbP \Sing(Y))}
 \Dmod( \bbP \fs_f^{-1}(O_X)).
$$
Then the assertion follows by plugging in \eqref{eqn:relation with AG2}.
\end{proof}

\begin{rem}
In the above proposition, we do not require that $X$ be a global complete intersection.
\end{rem}

\sssec{}

Let $W \xto{e} X \xto{f} Y$ be a string of quasi-smooth schemes, with $Y$ a global complete intersection as always in this section. Observe first that there is a natural correspondence
\begin{equation} \label{eqn:correspondece for a cor}
\fs_f^{-1}(O_X) \leftto W \times_X \fs_f^{-1}(O_X) \hto  \fs_{f \circ e}^{-1}(O_W),
\end{equation}
where we emphasize that the right arrow $W \times_X \fs_f^{-1}(O_X) \hto  \fs_{f \circ e}^{-1}(O_W)$ is a closed embedding. Observe also that the pullback functor $\xi^!: \ICoh(Y^\wedge_X) \to \ICoh(Y^\wedge_W)$ preserves the $\ICoh_0$-subcategories, that is, it restricts to a functor 
$$
\xi^{!,0}:
\ICoh_0(Y^\wedge_X) \longto
\ICoh_0(Y^\wedge_W).
$$ 

\begin{cor} 
Under the equivalences
\begin{eqnarray}
& & 
\ICoh_0(Y^\wedge_X)  
\simeq
\ICoh(Y)
\usotimes{\DSingY }
\Dmod( \fs_f^{-1}(O_X))^\unshift \\
& & 
\ICoh_0(Y^\wedge_W) 
\simeq
\ICoh(Y)
\usotimes{\DSingY  }
\Dmod( \fs_{f \circ e}^{-1}(O_W))^\unshift
\end{eqnarray}
of the above proposition, the pullback functor $\xi^{!,0}$ is induced by the $\Dmod^\unshift$-module pull-push along the correspondence \eqref{eqn:correspondece for a cor}.
\end{cor}

\begin{proof}
This is simply because the equivalence
$$
\ICoh(Y^\wedge_X)
\simeq
\ICoh(Y)
\usotimes{\DSingY}
\Dmod(X \times_Y \Sing(Y))^\unshift
$$
is functorial in $X$ under pullbacks. 
\end{proof}

\sec{Strong spectral gluing} \label{sec:strong spectral gluing}

In this section, we construct our glued DG category
$$
\lim_{[\PQ] \in \Tw(\Par)^\op}
\,
\ICoh_0((\LSG)^\wedge_{\LSQ})^{P-\temp}
$$
and the gluing functor
$$
\ICoh_\N(\LSG)
\longto
\lim_{[\PQ] \in \Tw(\Par)^\op}
\,
\ICoh_0((\LSG)^\wedge_{\LSQ})^{P-\temp}.
$$
The construction follows a general paradigm, which we eventually apply to the case of local systems.

\ssec{Some preliminary constructions}

\sssec{}

For $f:X \to Y$ a map of quasi-smooth stacks, denote by $X \xto{'f} Y^\wedge_X \xto{\wh f} Y$ the associated factorization through the formal completion.
The inclusion
$$
\Xi_{Y^\wedge_X}: \QCoh(Y^\wedge_X)
\hto
\ICoh(Y^\wedge_X)
$$
obviously factors as
$$
 \QCoh(Y^\wedge_X)
 \hto
  \ICoh_0(Y^\wedge_X)
  \hto
   \ICoh(Y^\wedge_X).
$$
Abusing notation, the first inclusion $\QCoh(Y^\wedge_X) \hto \ICoh_0(Y^\wedge_X)$ will be denoted by $\Xi_{Y^\wedge_X}$ as well. 
Similarly, the right adjoint $\Psi_{Y^\wedge_X}: \ICoh(Y^\wedge_X) \tto
\QCoh(Y^\wedge_X)$ induces a right adjoint $\ICoh_0(Y^\wedge_X) \tto
\QCoh(Y^\wedge_X)$, denoted by the same symbol.

\begin{rem}
Let us show that the above functor
$$
\Psi_{Y^\wedge_X}: \ICoh_0(Y^\wedge_X)  \longto \QCoh(Y^\wedge_{X}) 
$$
is $\Dmod(X)$-linear. This fact will be used implicitly later, especially at the end of Lemma \ref{lem: right adjojntables}. 
Factoring $\Psi$ as
$$
\ICoh_0(Y^\wedge_X) 
\hto
\ICoh(Y^\wedge_X) 
 \tto \QCoh(Y^\wedge_{X})
$$
and noticing that the left arrow is $\Dmod(X)$-linear, it suffices to treat the second arrow. As the latter is right adjoint to a $\Dmod(X)$-linear functor, it is a priori lax $\Dmod(X)$-linear. To verify that such lax linearity is actually strict, it suffices to work smooth-locally on $Y$. Thus, we may assume that $Y$ is a scheme. In this case, we are dealing with the functor
$$
\ICoh(Y) \usotimes{\Dmod(Y)} \Dmod(X)
\xto{\Psi_Y \otimes \id_{\Dmod(X)}}
\QCoh(Y) \usotimes{\Dmod(Y)} \Dmod(X),
$$
which is evidently $\Dmod(X)$-linear.
\end{rem}

\sssec{~A technical note.}

Recall the standard functor $\Upsilon$ defined in \cite{ICoh} and \cite{Book}.
There are two possible realizations of $\ICoh_0(Y^\wedge_X)$: the one of \cite{AG2} uses the embedding $\Xi_X: \QCoh(X) \hto \ICoh(X)$, while the one of \cite{centerH} uses the embedding $\Upsilon_X: \QCoh(X) \to \ICoh(X)$. For a moment, let us denote the former by $\ICoh_0^{(\Xi)}(Y^\wedge_X)$ and the latter by $\ICoh_0^{(\Upsilon)}(Y^\wedge_X)$. These two DG categories are obviously equivalent (indeed, for $X$ is quasi-smooth, the functors $\Xi_X$ and $\Upsilon_X$ differ only by a shifted line bundle), however their functoriality under pullbacks is slightly different.
To be consistent with \cite{AG2}, we use the $\Xi$-realization. For this reason, the functoriality of $\ICoh_0$ developed in \cite{centerH} must be reinterpreted accordingly, that is, by conjugating with the natural equivalences $\sigma: \ICoh_0^{(\Xi)} \xto{\simeq} \ICoh_0^{(\Upsilon)}$ of \cite[Proposition 3.2.4]{centerH}.
Practically, all we need to know is that such $\sigma_{Y^\wedge_X}$ intertwines the inclusion $\Xi_{Y^\wedge_X}$ with the inclusion $\Upsilon_{Y^\wedge_X}$.

\sssec{}

The $\H(Y)$-linear pullback functor $(\wh f)^{!,0}: \QCoh(Y) \simeq \ICoh_0(Y^\wedge_Y) \to \ICoh_0(Y^\wedge_X)$ factors as the composition
\begin{equation} 	\nonumber
\begin{tikzpicture}[scale=1.5]
\node (a) at (0,1) {$\QCoh(Y)$};
\node (b) at (2,1) {$\QCoh(Y^\wedge_X)$};
\node (c) at (4,1) {$\ICoh_0(Y^\wedge_X)$};
\path[->,font=\scriptsize,>=angle 90]
([yshift= 1.5pt]a.east) edge node[above] {$(\wh f)^*$} ([yshift= 1.5pt]b.west);
\path[right hook ->,font=\scriptsize,>=angle 90]
([yshift= 1.5pt]b.east) edge node[above] {$ \Xi_{Y^\wedge_X}$} ([yshift= 1.5pt]c.west);
\end{tikzpicture}
\end{equation}
of two $\H(Y)$-linear functors. We thus have an $\H(Y)$-linear diagram
\begin{equation} 	\nonumber
\begin{tikzpicture}[scale=1.5]
\node (a) at (0,1.2) {$\QCoh(Y) $};
\node (b) at (0,0) {$\QCoh(Y^\wedge_X)$};
\node (c) at (3,0) {$\ICoh_0(Y^\wedge_X)$,};
\path[->,font=\scriptsize,>=angle 90]
([yshift= 1.5pt]a.south) edge node[left] { $(\wh f)^{*}$} ([yshift= 1.5pt]b.north);
\path[right hook ->,font=\scriptsize,>=angle 90]
([yshift= 1.5pt]b.east) edge node[above] {$ \Xi_{Y^\wedge_X}$} ([yshift= 1.5pt]c.west);
\path[->>,font=\scriptsize,>=angle 90]
([yshift= -1.5pt]c.west) edge node[below] {$\Psi_{Y^\wedge_X}$} ([yshift= -1.5pt]b.east);
\path[->,font=\scriptsize,>=angle 90]
([yshift= 1.5pt]a.east) edge node[above] { $(\wh f)^{!,0}$} ([yshift= 1.5pt]c.north);
\end{tikzpicture}
\end{equation}
with both triangles commutative.

\sssec{}

Assume now that $X  \xto f Y$ is a map of quasi-smooth stacks over a third quasi-smooth stack $Z$. Denote by $\wt f: Z^\wedge_X \to Z^\wedge_Y$ the induced arrow. We will tensor up the above diagram with $\Hsx ZY$ over $\H(Y)$, see Section \ref{sssec:Review H}.
Thanks to the two natural equivalences\footnote{these are instances of the main theorem of \cite{shvcatHH}, see also the fourth item of the list in Section \ref{sssec:Review H}}
$$
\Hsx ZY \usotimes{\H(Y)} \QCoh(Y)
\simeq
\ICoh_0(Z^\wedge_Y)
$$
$$
\Hsx ZY \usotimes{\H(Y)} \ICoh_0(Y^\wedge_X)
\simeq
\ICoh_0(Z^\wedge_X),
$$
we obtain the diagram
\begin{equation}
\begin{minipage}{0.5\linewidth}
\begin{tikzpicture}[scale=1.5] \label{diag:triangle with mixed pullback}
\node (a) at (0,1.2) {$\ICoh_0(Z^\wedge_Y)$};
\node (b) at (0,0) {$\Hsx ZY \otimes_{\H(Y)}\QCoh(Y^\wedge_X)$};
\node (c) at (3.7,0) {$\ICoh_0(Z^\wedge_X)$,};
\path[->,font=\scriptsize,>=angle 90]
([yshift= 1.5pt]a.south) edge node[left] { mixed pullback} ([yshift= 1.5pt]b.north);
\path[right hook ->,font=\scriptsize,>=angle 90]
([yshift= 1.5pt]b.east) edge node[above] {$\id \otimes \Xi_{Y^\wedge_X}$} ([yshift= 1.5pt]c.west);
\path[->>,font=\scriptsize,>=angle 90]
([yshift= -1.5pt]c.west) edge node[below] {$\id \otimes \Psi_{Y^\wedge_X}$} ([yshift= -1.5pt]b.east);
\path[->,font=\scriptsize,>=angle 90]
([yshift= 1.5pt]a.east) edge node[above] { $(\wt f)^{!,0}$}
 ([yshift= 1.5pt]c.north);
\end{tikzpicture}
\end{minipage}
\end{equation}
with both triangles again commutative.

\sssec{}

We need a notation for the DG category appearing on the bottom left of the above diagram and for the vertical arrow. We set:
$$
  \ICoh_0(Z^\wedge_X)^{Y-\temp}
  :=
\Hsx ZY \otimes_{\H(Y)}\QCoh(Y^\wedge_X).
$$
When regarding $\ICoh_0(Z^\wedge_X)^{Y-\temp}$ as a full subcategory of $ \ICoh_0(Z^\wedge_X)$, we refer to it as the \emph{$Y$-tempered} subcategory. 
The next lemma gives an explicit characterization of $  \ICoh_0(Z^\wedge_X)^{Y-\temp}$.

\begin{lem} \label{lem:easy charact of ICoh zero temp.}
The inclusion $  \ICoh_0(Z^\wedge_X)^{Y-\temp} \hto   \ICoh_0(Z^\wedge_X)$ restricts to the equivalence
$$
  \ICoh_0(Z^\wedge_X)^{Y-\temp}
\simeq
    \ICoh_0(Z^\wedge_X)
    \ustimes
    {
      \ICoh_0(Y^\wedge_X)
    }
       \QCoh(Y^\wedge_X).
$$

\end{lem}

\begin{proof}
We proceed as in \cite[Corollary 3.2.5]{AG2}. A routine base-change computation shows that the adjunction
\begin{equation} 	\nonumber
\begin{tikzpicture}[scale=1.5]
\node (a) at (0,1) {$ \ICoh_0(Y^\wedge_X) $};
\node (b) at (3,1) {$\Hsx ZY \otimes_{\H(Y)} \ICoh_0(Y^\wedge_X)$};
\path[ ->,font=\scriptsize,>=angle 90]
([yshift= 1.5pt]a.east) edge node[above] {$ $ } ([yshift= 1.5pt]b.west);
\path[->,font=\scriptsize,>=angle 90]
([yshift= -1.5pt]b.west) edge node[below] {$ $ } ([yshift= -1.5pt]a.east);
\end{tikzpicture}
\end{equation}
induced by $\H(Y) \rightleftarrows \Hsx ZY$ goes over, 
under the equivalence
$$
\Hsx ZY \otimes_{\H(Y)} \ICoh_0(Y^\wedge_X)
\simeq
\ICoh_0(Z^\wedge_X),
$$
to the monadic adjunction
\begin{equation} 	\nonumber
\begin{tikzpicture}[scale=1.5]
\node (a) at (0,1) {$ \ICoh_0(Y^\wedge_X) $};
\node (b) at (3,1) {$ \ICoh_0(Z^\wedge_X)$.};
\path[ ->,font=\scriptsize,>=angle 90]
([yshift= 1.5pt]a.east) edge node[above] {$(\wt f)_*^{\ICoh}$ } ([yshift= 1.5pt]b.west);
\path[->,font=\scriptsize,>=angle 90]
([yshift= -1.5pt]b.west) edge node[below] {$(\wt f)^{!,0}$ } ([yshift= -1.5pt]a.east);
\end{tikzpicture}
\end{equation}
As a consequence of the computation, the functor $(\wt f)_*^{\ICoh}$ preserves the $\ICoh_0$-subcategories. 
It suffices to check that the latter adjunction restricts to an adjunction
\begin{equation} 	\nonumber
\begin{tikzpicture}[scale=1.5]
\node (a) at (0,1) {$ \QCoh(Y^\wedge_X) $};
\node (b) at (4,1) {$  \ICoh_0(Z^\wedge_X)
    \times_
    {
      \ICoh_0(Y^\wedge_X)
    }
       \QCoh(Y^\wedge_X).$};
\path[ ->,font=\scriptsize,>=angle 90]
([yshift= 1.5pt]a.east) edge node[above] { } ([yshift= 1.5pt]b.west);
\path[->,font=\scriptsize,>=angle 90]
([yshift= -1.5pt]b.west) edge node[below] {  } ([yshift= -1.5pt]a.east);
\end{tikzpicture}
\end{equation}
This is clear for the right adjoint. As for the left adjoint, it suffices to prove that the monad $(\wt f)^! (\wt f)_*^{\ICoh}$ preserves $\QCoh(Y^\wedge_X)$. The monad in question admits a nonnegative filtration with associated graded being the functor of tensoring with $\Sym (\Tang^\QCoh_{Y/Z})$, in the sense of the action of $\QCoh(Y)$ on $\ICoh(Y^\wedge_X)$ by pullback. The assertion follows. 
\end{proof}

\sssec{}

We denote by
$$
(\wt f)^{!,Y-\temp}:
\ICoh_0(Z ^\wedge_Y)
\longto
  \ICoh_0(Z^\wedge_X)^{Y-\temp}
$$
the vertical arrow appearing in \eqref{diag:triangle with mixed pullback}. Any map $f: X \to Y$ of quasi-smooth stacks over $Z$ yields a cospan
$$
 \ICoh_0(Z^\wedge_Y)
\xto{ (\wt f)^{!, Y-\temp}  }
  \ICoh_0(Z^\wedge_X)^{Y-\temp}
\xleftarrow{\;\;\Psi_{Z,f} \;\;}
  \ICoh_0(Z^\wedge_X)
$$
in $\H(Z) \mmod$, where we have set $\Psi_{Z,f} := \id_{\Hsx ZY} \otimes \Psi_{Y^\wedge_X}$.

\sssec{} \label{sssec:intermediate cat more familiar}
If $Y$ happens to be a scheme, then $\ICoh_0(Z^\wedge_X)^{Y-\temp}$ can be rewritten more simply as
$$
  \ICoh_0(Z^\wedge_X)^{Y-\temp}
\simeq
 \ICoh_0(Z^\wedge_Y)
 \usotimes{\Dmod(Y)}
 \Dmod(X).
$$
This follows from the natural $\H(Y)$-linear equivalence
$$
\QCoh(Y^\wedge_X) 
\simeq 
\QCoh(Y) 
\usotimes{\Dmod(Y)}
\Dmod(X),
$$
which is valid when $Y$ is a scheme, together with the $(\H(Y), \Dmod(Y))$-bimodule structure of $\QCoh(Y)$.
Thus, in this case, the pullback $(\wt f)^{!,Y-\temp}$ is obtained by the $\fD$-module pullback $\Dmod(Y) \to \Dmod(X)$ upon tensoring up.

\ssec{Constructing the glued DG category}

We retain the notation from the previous section.

\sssec{} \label{sssec:functoriality quesiton}

Regarding the quasi-smooth stack $Z$ as fixed throughout, let us consider the full subcategory $\QSmooth_{/Z} \subseteq \Stk_{/Z}$ spanned by quasi-smooth stacks mapping schematically to $Z$.
We need to study the functoriality of the assignment
\begin{equation} \label{eqn:assignemnt of cospan to arrow}
[X \xto f Y]
\squigto
\Big[ \ICoh_0(Z^\wedge_Y)
\xto{ (\wt f)^{!, Y-\temp}  }
  \ICoh_0(Z^\wedge_X)^{Y-\temp}
\xleftarrow{\;\;\Psi_{Z, f} \;\;}
  \ICoh_0(Z^\wedge_X)
  \Big]
\end{equation}
from arrows in $\QSmooth_{/Z}$ to cospans in $\H(Z) \mmod$.

\sssec{} \label{sssec:twisted arrows}

To this end, recall the $\infty$-category $\Tw(\I)$ of \emph{twisted arrows} associated to an $\infty$-category $\I$, see \cite[Section 4.2]{DAGX}. 
Our convention is that 
$$
\Tw(\I) = \{[i \to j]\}
$$
is covariant in the first argument and contravariant in the second argument.
The answer to the functoriality question posed in Section \ref{sssec:functoriality quesiton} is that there is a functor (to be constructed below)
$$
\ICoh_0^\Tw: 
(\Tw(\QSmooth_{/Z} ))^\op 
\longto
\H(Z) \mmod
$$
that sends the correspondence
$$
[Y \to Y] \longleftarrow [X \to Y] \longto [X \to X]
$$
in $\Tw(\QSmooth_{/Z})$ to the cospan appearing in \eqref{eqn:assignemnt of cospan to arrow}.

\sssec{}

To construct $\ICoh_0^\Tw$, we need an auxiliary notion: the simplicial space of \emph{descending grids} associated to an $\infty$-category $\E$. 
For $n \geq 0$, let 
$$
\Grid^{\downarrow}_n(\E) := \Maps
\Bigt{
([n]^\op \times [n])^{\leq}
, 
\E
},
$$
where $([n]^\op \times [n])^{\leq}$ is the poset
$$
([n]^\op \times [n])^{\leq} := \{(i,j) \in [n]^\op \times [n] \, \big| \, i \leq j   \}.
$$
Pictorially, $\Grid^{\downarrow}_n(\E)$ is the space of commutative diagrams (henceforth called descending $n$-grids)
\begin{equation}  \label{diag:big grid}
\begin{minipage}{0.45\linewidth}
\begin{tikzpicture}[scale=1.5]
\node (54) at (5.5,2.8) {$x_{0,0}$};
\node (53) at (5.5,2.1) {$x_{0,1}$};
\node (52) at (5.5,1.4) {$\vdots$};
\node (51) at (5.5,.7) {$x_{0,n-1}$};
\node (50) at (5.5,0) {$x_{0,n}$};
\node (43) at (4.4,2.1) {$x_{1,1}$};
\node (42) at (4.4,1.4) {$\vdots$};
\node (41) at (4.4,.7) {$x_{1,n-1}$};
\node (40) at (4.4,0) {$x_{1,n}$};
\node (32) at (3.3,1.4) {$\ddots$};
\node (31) at (3.3,.7) {$\cdots$};
\node (30) at (3.3,0) {$\cdots$};
\node (21) at (2.2,.7) {$x_{n-1,n-1}$};
\node (20) at (2.2,0) {$x_{n-1,n}$};
\node (10) at (1,0) {$x_{n,n}$};
\path[<-,font=\scriptsize,>=angle 90]
(20.north) edge node[above] {$ $} (21.south); 
\path[<-,font=\scriptsize,>=angle 90]
(30.north) edge node[above] {$ $} (31.south); 
\path[<-,font=\scriptsize,>=angle 90]
(40.north) edge node[above] {$ $} (41.south); 
\path[<-,font=\scriptsize,>=angle 90]
(50.north) edge node[above] {$ $} (51.south); 
\path[<-,font=\scriptsize,>=angle 90]
(31.north) edge node[above] {$ $} (32.south); 
\path[<-,font=\scriptsize,>=angle 90]
(41.north) edge node[above] {$ $} (42.south); 
\path[<-,font=\scriptsize,>=angle 90]
(51.north) edge node[above] {$ $} (52.south); 
\path[<-,font=\scriptsize,>=angle 90]
(42.north) edge node[above] {$ $} (43.south); 
\path[<-,font=\scriptsize,>=angle 90]
(52.north) edge node[above] {$ $} (53.south); 
\path[<-,font=\scriptsize,>=angle 90]
(53.north) edge node[above] {$ $} (54.south); 
\path[->,font=\scriptsize,>=angle 90]
(10.east) edge node[above] {$ $}  (20.west);
\path[->,font=\scriptsize,>=angle 90]
(20.east) edge node[above] {$ $}  (30.west);
\path[->,font=\scriptsize,>=angle 90]
(30.east) edge node[above] {$ $}  (40.west);
\path[->,font=\scriptsize,>=angle 90]
(40.east) edge node[above] {$ $}  (50.west);
\path[->,font=\scriptsize,>=angle 90]
(21.east) edge node[above] {$ $}  (31.west);
\path[->,font=\scriptsize,>=angle 90]
(31.east) edge node[above] {$ $}  (41.west);
\path[->,font=\scriptsize,>=angle 90]
(41.east) edge node[above] {$ $}  (51.west);
\path[->,font=\scriptsize,>=angle 90]
(32.east) edge node[above] {$ $}  (42.west);
\path[->,font=\scriptsize,>=angle 90]
(42.east) edge node[above] {$ $}  (52.west);
\path[->,font=\scriptsize,>=angle 90]
(43.east) edge node[above] {$ $}  (53.west);
\end{tikzpicture}
\end{minipage}
\end{equation} 
of objects of $\E$. The assignment $[n] \squigto \Grid^{\downarrow}_n(\E)$ is naturally a simplicial space, which we will denote by $\Grid^{\downarrow}_\bullet(\E)$.

\sssec{}

Similarly, we have the notion of \emph{ascending grid}, obtained from the above by reversing (only) the vertical arrows. Precisely, we consider
$$
\Grid^{\uparrow}_n(\E) := \Maps
\Bigt{
([n]^\op \times [n]^\op)^{\leq}
, 
\E
},
$$
where $([n]^\op \times [n]^\op)^{\leq}$ is the poset
$ \{(i,j) \in [n]^\op \times [n]^\op \, \big| \, i \leq j   \}$.

\sssec{}

By a slight modification of \cite[Appendix B]{chiralcat}, we find that the space of $\infty$-functors $\Tw(\I)^\op \to \E$ is equivalent to the space of maps  $N_\bullet(\I) \to \Grid^\downarrow_\bullet(\E)$ of simplicial spaces, where $N_\bullet(\I)$ is the nerve of $\I$.
We use this to define our functor $\ICoh_0^\Tw: \Tw(\QSmooth_{/Z})^\op \to \H(Z) \mmod$: we will provide a simplicial assignment\footnote{Formally, one should use the Grothendieck construction associated to $\Grid_\bullet(\H(Z)\mmod)$. We leave these details to the reader.}
 of a descending $n$-grid $\grid^{\downarrow}_n(X_i)$ in $\H(Z)\mmod$ to any $[n] \in \bDelta$ and any string $X_0 \to X_1 \to \cdots \to X_n$ of quasi-smooth stacks mapping schematically to $Z$. 

\sssec{}

To construct $\grid_n^{\downarrow}(X_i)$, we proceed in two steps. We first construct an ascending version $\grid_n^{\uparrow}(X_i)$, which is easier to handle. Secondly, we prove that each of the squares forming $\grid_n^{\uparrow}(X_i)$ is vertically right adjointable: this means that the vertical arrows all admit continuous (and automatically $\H(Z)$-linear, since $\H(Z)$ is rigid) right adjoints and that the resulting lax-commutative squares are commutative.
This yields a well-defined descending grid $\grid_n^{\downarrow}(X_i)$, and it will be clear that the assignment
$$([n], X_0 \to \cdots X_n) \squigto \grid_n^{\downarrow}(X_i)$$ determines a map of simplicial spaces as desired.

\sssec{} \label{sssec:assignment of grid}
 
Let us now define $\grid^\uparrow_n(X_i)$.
\begin{itemize}
\item
The DG category in the entry $(i,j)$ is
$$
\C_{i,j} := \ICoh_0 \bigt{ Z^\wedge_{X_i} }^{X_j -\temp}.
$$

\smallskip

\item
For $i \leq j \leq k$, the horizontal arrow corresponding to $(j, k) \leftto (i,k)$ is the pullback
$$
\ICoh_0(Z^\wedge_{X_j})^{X_k -\temp}
\to
\ICoh_0(Z^\wedge_{X_i})^{X_k -\temp}
$$
induced by the $(!,0)$-pullback along $Z^\wedge_{X_i} \to Z^\wedge_{X_j}$. Thanks to Lemma \ref{lem:easy charact of ICoh zero temp.}, such pullback does indeed preserve the $X_k$-tempered subcategories.

\smallskip

\item
For $i \leq j \leq k$, the vertical arrow corresponding to $(i, k) \leftto (i,j)$ is the obvious inclusion
$$
\ICoh_0(Z^\wedge_{X_i})^{X_k -\temp}
\hto
\ICoh_0(Z^\wedge_{X_i})^{X_j -\temp}.
$$
\end{itemize}
In view of the functoriality of $\ICoh_0$ under pullbacks, this datum underlies an ascending $n$-grid. In fact, for each $i \leq j$, the $(!,0)$-pullback along $Z^\wedge_{X_i} \to Z^\wedge_{X_j}$ gives the diagonal arrow $(j,j) \to (i,i)$ pointing north-east, and it is clear that these diagonal arrows, together with the vertical ascending inclusions, determine the rest of the grid uniquely.

\sssec{} 

Consider the arrow $N_\bullet(\QSmooth_{/Z}) \to \Grid^{\uparrow}_\bullet(\H(Z)\mmod)$ just constructed, schematically denoted by $([n], X_i) \squigto \grid_n^{\uparrow}(X_i)$. A string $i \leq j \leq k \leq \ell$ determines the following commutative square 
\begin{equation} 
\nonumber
\begin{tikzpicture}[scale=1.5]
\node (00) at (0,0) {$\ICoh_0(Z^\wedge_{X_j})^{X_\ell -\temp}$};
\node (10) at (3.5,0) {$\ICoh_0(Z^\wedge_{X_i})^{X_\ell -\temp}$,};
\node (01) at (0,1) {$\ICoh_0(Z^\wedge_{X_j})^{X_k -\temp}$};
\node (11) at (3.5,1) {$\ICoh_0(Z^\wedge_{X_i})^{X_k -\temp}$};
\path[->,font=\scriptsize,>=angle 90]
(00.east) edge node[above] {$\mbox{pullback}$} (10.west); 
\path[ ->,font=\scriptsize,>=angle 90]
(01.east) edge node[above] {$\mbox{pullback}$} (11.west); 
\path[<- right hook ,font=\scriptsize,>=angle 90]
(01.south) edge node[right] {$\Xi$} (00.north);
\path[<- right hook,font=\scriptsize,>=angle 90]
(11.south) edge node[right] {$\Xi$} (10.north);
\end{tikzpicture}
\end{equation} 
which appears as a general square of $\grid_n^{\uparrow}(X_i)$.
In the following lemma, we prove the promised property of such a square.

\begin{lem} \label{lem: right adjojntables}
The above square is vertically right adjointable, meaning that the vertical arrows admit continuous right adjoints and that the lax commutative square obtained by changing the vertical arrows with these right adjoints is commutative.
\end{lem}

\begin{proof}
The assertion about the continuity of the right adjoints, to be denoted simply by $\Psi$, is clear. Indeed, 
$$
\Xi: \ICoh_0(Z^\wedge_{X_j})^{X_\ell -\temp} 
\hto \ICoh_0(Z^\wedge_{X_j})^{X_k -\temp}
$$
is obtained from 
$$
\Xi: \QCoh((X_\ell)^\wedge_{X_j})
\hto \ICoh_0((X_\ell)^\wedge_{X_j})^{X_k -\temp}
$$
by tensoring up with $\Hsx Z{X_\ell}$ over $\H(X_\ell)$, and the latter functor evidently admits a continuous right adjoint, given by the composition 
\begin{equation} 
\nonumber
\begin{tikzpicture}[scale=1.5]
\node (00) at (0,0) {$
 \ICoh_0((X_\ell)^\wedge_{X_j})^{X_k -\temp}$};
\node (10) at (3,0) {$\ICoh_0((X_\ell)^\wedge_{X_j})
$};

\node (20) at (6,0) {$ \QCoh((X_\ell)^\wedge_{X_j})
$.};
\path[right hook ->,font=\scriptsize,>=angle 90]
(00.east) edge node[above] {$\Xi$} (10.west); 
\path[->>,font=\scriptsize,>=angle 90]
(10.east) edge node[above] {${\Psi_{(X_\ell)^\wedge_{X_j}}}$} (20.west); 
\end{tikzpicture}
\end{equation} 
It remains to prove that the lax commutative square
\begin{equation} 
\nonumber
\begin{tikzpicture}[scale=1.5]
\node (00) at (0,0) {$\ICoh_0(Z^\wedge_{X_j})^{X_\ell -\temp}$};
\node (10) at (3.5,0) {$\ICoh_0(Z^\wedge_{X_i})^{X_\ell -\temp}$};
\node (01) at (0,1) {$\ICoh_0(Z^\wedge_{X_j})^{X_k -\temp}$};
\node (11) at (3.5,1) {$\ICoh_0(Z^\wedge_{X_i})^{X_k -\temp}$};
\path[->,font=\scriptsize,>=angle 90]
(00.east) edge node[above] {$\mbox{pullback}$} (10.west); 
\path[ ->,font=\scriptsize,>=angle 90]
(01.east) edge node[above] {$\mbox{pullback}$} (11.west); 
\path[->> ,font=\scriptsize,>=angle 90]
(01.south) edge node[right] {$\Psi$} (00.north);
\path[->>,font=\scriptsize,>=angle 90]
(11.south) edge node[right] {$\Psi$} (10.north);
\end{tikzpicture}
\end{equation} 
is commutative. As a first simplication, we can assume that the map $X_\ell = Z$ is an isomorphism: indeed the above square is obtained from
\begin{equation} 
\nonumber
\begin{tikzpicture}[scale=1.5]
\node (00) at (0,0) {$\QCoh((X_\ell)^\wedge_{X_j})$};
\node (10) at (3.8,0) {$\QCoh((X_\ell)^\wedge_{X_i})$};
\node (01) at (0,1) {$\ICoh_0((X_\ell)^\wedge_{X_j})^{X_k -\temp}$};
\node (11) at (3.8,1) {$\ICoh_0((X_\ell)^\wedge_{X_i})^{X_k -\temp}$};
\path[->,font=\scriptsize,>=angle 90]
(00.east) edge node[above] {$\mbox{pullback}$} (10.west); 
\path[ ->,font=\scriptsize,>=angle 90]
(01.east) edge node[above] {$\mbox{pullback}$} (11.west); 
\path[->> ,font=\scriptsize,>=angle 90]
(01.south) edge node[right] {$\Psi$} (00.north);
\path[->>,font=\scriptsize,>=angle 90]
(11.south) edge node[right] {$\Psi$} (10.north);
\end{tikzpicture}
\end{equation} 
by tensoring up with $\Hsx Z{X_\ell}$ over $\H(X_\ell)$. By pulling back to an atlas of $X_\ell$, we can assume that all the stacks in question are schemes (recall that all the maps to $Z = X_\ell$ are schematic). Then, thanks to the observation of Section \ref{sssec:intermediate cat more familiar}, the latter diagram can be rewritten as
\begin{equation} 
\nonumber
\begin{tikzpicture}[scale=1.5]
\node (00) at (0,0) 
{$
\QCoh((X_\ell)^\wedge_{X_k}) 
\usotimes{\Dmod(X_k)} 
\Dmod(X_j)
$};
\node (10) at (5.5,0) 
{$
\QCoh((X_\ell)^\wedge_{X_k}) 
\usotimes{\Dmod(X_k)} 
\Dmod(X_i).
$};
\node (01) at (0,1) 
{$
\ICoh_0((X_\ell)^\wedge_{X_k}) 
\usotimes{\Dmod(X_k)} 
\Dmod(X_j)
$};
\node (11) at (5.5,1)
{$
\ICoh_0((X_\ell)^\wedge_{X_k}) 
\usotimes{\Dmod(X_k)} 
\Dmod(X_i)
$};
\path[->,font=\scriptsize,>=angle 90]
(00.east) edge node[above] {$\id \otimes (X_i \to X_j)^{!,\dR}$} (10.west); 
\path[ ->,font=\scriptsize,>=angle 90]
(01.east) edge node[above] {$\id \otimes (X_i \to X_j)^{!,\dR}$} (11.west); 
\path[->> ,font=\scriptsize,>=angle 90]
(01.south) edge node[right] {$\Psi \otimes \id$} (00.north);
\path[->>,font=\scriptsize,>=angle 90]
(11.south) edge node[right] {$\Psi \otimes \id$} (10.north);
\end{tikzpicture}
\end{equation} 
This diagram is manifestly commutative, as the vertical and horizontal arrows are \virg{decoupled}.
\end{proof}

\ssec{The gluing functor and the statement of the main theorem}

In this section, we use the above construction to define our gluing functor and state our main theorem.

\sssec{}

In the previous section, we have constructed a functor $\ICoh_0^{\Tw}: \Tw(\QSmooth_{/Z})^\op \to \H(Z)\mmod$.
Let us now fix $Z := \LSG$ and consider the obvious functor
$$
\phi:
\Par \longto \QSmooth_{/\LSG}, 
\hspace{.4cm}
P \squigto \LSP,
$$
where $\Par$ denotes the poset of standard parabolic subgroups of $G$ (ordered by inclusion). Composition yields a functor
$$
\ICoh_{0}^{\Tw,\Par}: \Tw(\Par)^\op 
\longto 
\H(\LSG) \mmod
$$
$$
\hspace{3.9cm}
[\PQ]
\squigto
\ICoh_0((\LSG)^\wedge_{\LS_Q})^{P-\temp}
$$
out of the $1$-category of twisted arrows of the poset $\Par$. 
Here we are using the notation \virg{$P-\temp$} instead of the more cumbersome \virg{$\LSP- \temp$}.

\begin{rem} \label{rem:twisted arrows Par op}

The $1$-category $\Tw := \Tw(\Par)$ is easy to grasp: it is a poset consisting of $3^{\rank(G)}$ elements. For $G$ of semisimple rank one, $\Tw$ is simply a correspondence diagram
$$
[B \subseteq B] 
\longleftarrow
[B \subseteq G] 
\longto
[G \subseteq G]. 
$$
For $G$ of semisimple rank two, the Hasse diagram of $\Tw$ is 
$$ 
\begin{tikzpicture}[scale=1.5]
\node (00) at (0,0) {$[B \subseteq B] $};
\node (10) at (2,0) {$[B \subseteq Q] $};
\node (20) at (4,0) {$[Q \subseteq Q] $,};
\node (01) at (0,1) {$[B \subseteq P] $};
\node (11) at (2,1) {$[B \subseteq G]$};
\node (21) at (4,1) {$[Q \subseteq G] $};
\node (02) at (0,2) {$[P \subseteq P] $};
\node (12) at (2,2) {$[P \subseteq G] $};
\node (22) at (4,2) {$[G \subseteq G] $};
\path[->,font=\scriptsize,>=angle 90]
(12.east) edge node[below] {} (22.west);
\path[->,font=\scriptsize,>=angle 90]
(11.east) edge node[below] {$ $} (21.west);
\path[->,font=\scriptsize,>=angle 90]
(10.east) edge node[below] {$ $} (20.west);
\path[<-,font=\scriptsize,>=angle 90]
(02.east) edge node[below] { } (12.west);
\path[<-,font=\scriptsize,>=angle 90]
(01.east) edge node[below] {$ $} (11.west);
\path[<-,font=\scriptsize,>=angle 90]
(00.east) edge node[below] {$ $} (10.west);
\path[<-,font=\scriptsize,>=angle 90]
(02.south) edge node[right] {} (01.north);
\path[<-,font=\scriptsize,>=angle 90]
(12.south) edge node[below] {$ $} (11.north);
\path[<-,font=\scriptsize,>=angle 90]
(22.south) edge node[below] {$ $} (21.north);
\path[->,font=\scriptsize,>=angle 90]
(01.south) edge node[right] {} (00.north);
\path[->,font=\scriptsize,>=angle 90]
(11.south) edge node[below] {$ $} (10.north);
\path[->,font=\scriptsize,>=angle 90]
(21.south) edge node[below] {$ $} (20.north);
\end{tikzpicture}
$$
where $P$ and $Q$ denote the two maximal parabolic subgroups.
Like any category of twisted arrows, $\Tw$ has morphisms of two kinds: the ones that keep the first argument fixed (to be called morphisms of the \emph{first type}) and the ones that keep the second argument fixed (to be called morphisms of the \emph{second type}).
\end{rem}

\sssec{}

The limit DG category
$$
\lim(\ICoh_0^{\Tw,\Par})
=
\lim_{[\PQ] \in \Tw(\Par)^\op} \; 
\ICoh_0((\LSG)^\wedge_{\LS_Q})^{P-\temp}
$$ 
is the \emph{glued DG category} that will feature in the strong gluing theorem.
To complete the statement of that theorem, it remains to exhibit an $\H(\LSG)$-linear functor from $\ICoh_\N(\LSG)$ to $\lim(\ICoh_0^{\Tw,\Par})$.

\sssec{}

As a short digression, let us discuss the relation between $\lim(\ICoh_0^{\Tw,\Par})$ and the lax-limit of \cite[Sections 4.2-4.3]{AG2}. Recall that the latter is the lax-limit of the functor
$$
\ICoh_0((\LSG)^\wedge_{\LS_?}: 
\Par^\op \longto \H(\LSG)\mmod
$$
induced by the $(!,0)$-pullbacks.
We wish to construct a fully faithful embedding
\begin{equation} \label{eqn: nice inclusion of laxlim into lim}
\lim \bigt{ \ICoh_0^{\Tw,\Par} }
\hto
\laxlim \Bigt{ \ICoh_0((\LSG)^\wedge_{\LS_?} ) }.
\end{equation}
To this end, consider the cocartesian fibration 
$$
\pi: \Groth(\ICoh_0^{\Tw,\Par}) \longto \Tw^\op
$$
associated to $\ICoh_0^{\Tw,\Par}$ by the Grothendieck construction. Then our glued category is the DG category of cocartesian sections of $\pi$, while it is easy to see that
$$
\laxlim \Bigt{ \ICoh_0((\LSG)^\wedge_{\LS_?} ) }
$$
is equivalent to the DG category of sections of $\pi$ that are required to be cocartesian only over the morphisms of $\Tw$ of the first type.\footnote{Recall these are the arrows $[\PQ] \to [Q \subseteq P']$ that fix the first argument.}

\sssec{}

Having established the inclusion \eqref{eqn: nice inclusion of laxlim into lim}, we can proceed to define an $\H(\LSG)$-linear functor 
$$
\ICoh_\N(\LSG) \longto \lim \bigt{ \ICoh_0^{\Tw,\Par} }.
$$
By \cite[Section 4.2]{AG2}, the functors
$$
\ICoh(\LSG) 
\xto{pullback}
\ICoh((\LSG)^\wedge_{\LSP})
\tto
\ICoh_0((\LSG)^\wedge_{\LSP})
$$
assemble to an $\H(\LSG)$-linear functor
\begin{equation} \label{eqn: functor from ICoh to lax lim}
\ICoh(\LSG) \longto
\laxlim \Bigt{ \ICoh_0((\LSG)^\wedge_{\LS_?} ) }.
\end{equation}

\begin{lem}
The above functor factors through the inclusion \eqref{eqn: nice inclusion of laxlim into lim}, thereby yielding an $\H(\LSG)$-linear functor
\begin{equation} \label{eqn: functor from ICoh to lim Tw}
\ICoh(\LSG) \longto
\lim \bigt{ \ICoh_0^{\Tw,\Par} }.
\end{equation}
\end{lem}

\begin{proof}
Unraveling the definitions, it suffices to verify that, for $\PQ$, the obviously commutative square
\begin{equation} 
\nonumber
\begin{tikzpicture}[scale=1.5]
\node (00) at (0,0) {$\ICoh_0((\LSG)^\wedge_{\LSP})$};
\node (10) at (4,0) {$\ICoh_0((\LSG)^\wedge_{\LSQ})^{P-\temp}$};
\node (01) at (0,1) {$\ICoh((\LSG)^\wedge_{\LSP})$};
\node (11) at (4,1) {$\ICoh((\LSG)^\wedge_{\LSQ})$};
\path[->,font=\scriptsize,>=angle 90]
(00.east) edge node[above] {$\mbox{pullback}$} (10.west); 
\path[ ->,font=\scriptsize,>=angle 90]
(01.east) edge node[above] {$\mbox{pullback}$} (11.west); 
\path[<- right hook ,font=\scriptsize,>=angle 90]
(01.south) edge node[right] {$ $} (00.north);
\path[<- right hook,font=\scriptsize,>=angle 90]
(11.south) edge node[right] {$ $} (10.north);
\end{tikzpicture}
\end{equation} 
is vertically right adjointable, see Lemma \ref{lem: right adjojntables}. The proof goes as in that lemma: pass to an atlas of $\LSG$ and use the observation of Section \ref{sssec:intermediate cat more familiar} to reduce to a decoupled diagram.  
\end{proof}

\sssec{}

By pre-composing with the inclusion $\ICoh_\N(\LSG) \hto \ICoh(\LSG)$, we finally get our gluing functor
$$
\gamma^{\strong}: 
\ICoh_\N(\LSG) \longto
\lim \bigt{ \ICoh_0^{\Tw,\Par} }.
$$
By construction, $\gamma^{\strong}$ is assembled out of the $\H(\LSG)$-linear functors
$$
\gamma_{[\PQ]}:
\ICoh_\N(\LSG)
\longto
\ICoh((\LSG)^\wedge_{\LSQ})
\tto
\ICoh_0((\LSG)^\wedge_{\LSQ})
\tto
\ICoh_0((\LSG)^\wedge_{\LSQ})^{P-\temp},
$$
for any arrow $[\PQ]$ in $\Tw$.
Our main theorem reads:

\begin{thm}[Strong spectral gluing] \label{mainthm:strong-ICOHN}
The $\H(\LSG)$-linear functor
\begin{equation} \label{equiv:main-ICOHN as limit of twisted arrow}
\gamma^\strong:
\ICoh_\N(\LSG) 
\longto
\lim \bigt{ \ICoh_0^{\Tw, \Par} }
=
\lim_{[\PQ] \in \Tw^\op}
\ICoh_0((\LSG)^\wedge_{\LSQ})^{P-\temp}
\end{equation} 
is an equivalence.
\end{thm}

\begin{example}
The DG category corresponding to $[G \supseteq Q]$ is 
$$
\ICoh_0((\LSG)^\wedge_{\LSQ})^{G-\temp}
\simeq
\QCoh((\LSG)^\wedge_{\LSQ}).
$$
Thus, for $G$ of semisimple rank $1$, the theorem states that the commutative diagram
\begin{equation} 
\nonumber
\begin{tikzpicture}[scale=1.5]
\node (00) at (0,0) {$\QCoh(\LSG)$};
\node (10) at (3,0) {$\QCoh((\LSG)^\wedge_{\LSB})$};
\node (01) at (0,1.2) {$\ICoh_\N(\LSG)$};
\node (11) at (3,1.2) {$\ICoh_0((\LSG)^\wedge_{\LSB})$};
\path[->,font=\scriptsize,>=angle 90]
(00.east) edge node[above] {$\mbox{pullback}$} (10.west); 
\path[->,font=\scriptsize,>=angle 90]
(01.east) edge node[above] {$ \gamma_{[B \subseteq B]}$} (11.west); 
\path[->,font=\scriptsize,>=angle 90]
(01.south) edge node[right] {$\Psi = \gamma_{[G \subseteq G]}$} (00.north);
\path[->,font=\scriptsize,>=angle 90]
(11.south) edge node[right] {$\Psi$} (10.north);
\end{tikzpicture}
\end{equation} 
is a fiber square. This is the diagram we have been considering in Theorem \ref{thm:intro-GL2 case}.
\end{example}

\sec{A $\fD$-module reformulation} \label{sec:microlocal reformulation}

To prove Theorem \ref{mainthm:strong-ICOHN}, we follow the strategy of \cite{AG2}. 
In the first step, we use the $\DSingY$-action on $\ICoh(Y)$, present for a complete intersection scheme $Y$, to reduce the statement to a gluing statement for DG categories of $\fD$-modules on various schemes of singularities. 
The second step, perfomed in Section \ref{sec:proof}, uses the combinatorics of \cite{AG2} together with two key results (Propositions \ref{prop:pseudo-proper-schemes} and \ref{prop:excision}) proven in Section \ref{ssec:key-tools}.

\ssec{Reducing to a simpler statement}

\sssec{}

For $\PQ$ two standard parabolic subgroups of $G$, consider the stack 
$$
\N_{\PQ}
:=
 \fs_{P}^{-1}(O_{\LSP}) \ustimes{\LSP} \LSQ,
$$ 
where by abuse of notation $\fs_P$ denotes the singular codifferential of $\p_P: \LSP \to \LSG$. In words, $\N_{\PQ}$ is the stack of pairs $(\sigma_Q, A)$, where $\sigma_Q$ is a $Q$-local system and $A$ a horizontal section of $(\u_P)_{\sigma_Q}$.

\sssec{}

The functoriality of shifted cotangent bundles implies that the assignment $[\PQ] \squigto \N_{\PQ}$ extends to a functor $\Tw := \Tw(\Par) \to \Stk$, given by compatible diagrams
\begin{equation} \label{corr:N PQ correspondence}
\N_{P \subseteq P } = \fs_P^{-1}(O_{\LSP})
\longleftarrow
\N_{\PQ} = 
 \fs_{P}^{-1}(O_{\LSP}) \ustimes{\LSP} \LSQ
\hto
\N_{Q \subseteq Q } = \fs_Q^{-1}(O_{\LSQ}).
\end{equation}
By pulling back along the natural maps 
$$
\mu_{\PQ}: \N_\PQ \longto \N,
$$
we obtain a $\Dmod(\Sing(\LSG))$-linear functor 
\begin{equation} \label{eqn: mu pullback for global nilp cone}
\mu^!: 
\Dmod(\N)
\longto
\lim_{[\PQ] \in \Tw^\op}
\Dmod(\N_{\PQ}).
\end{equation}

\begin{rem} \label{rem: passage to left adjoints}
Since the maps appearing in \eqref{corr:N PQ correspondence} are proper, the RHS above can also be written as a colimit by passing to the left adjoints of the transition functors.
Furthermore, since the maps $\mu_{\PQ}$ are all proper, $\mu^!$ admits a $\Dmod(\Sing(\LSG))$-linear left adjoint.
\end{rem}

\sssec{}

For a technical reason (the fact that $\N_\dR$ is not $1$-affine) that will appear evident later, we need to consider a slightly different functor: namely, we need to repeat the above construction after having pulled back to an atlas $Y \tto \LSG$.
In order to apply the theory of Section \ref{sec:global complete intersect}, we will choose an atlas that is a global complete intersection.
To do so explicitly, let us fix a point $x \in X$ once and for all. The choice of $x \in X$ gives rise to a map $\LSG \to BG$ and, since $X$ is assumed to be connected, to a canonical atlas 
$$
Y
:=
\LSG \ustimes{BG} \pt.
$$
Even more explicitly, we can write $Y = Y' \times_{\g} \pt$, where $Y'$ is a certain smooth subscheme of $\LS_G^{\on{RS}} \times_{\g/G} \g$, see \cite[Section 10.6]{AG1} for more details.

\sssec{}

Denote by $Y_Q$ and $N$ the schemes obtained from $\LSQ$ and $\N$ by pulling back along the map $Y \tto \LSG$.
We have
$$
Y_Q
:=
\LSQ \ustimes{\LSG} Y
\simeq
\LSQ \ustimes{BQ} G/Q,
$$
We set
$$
N_{\PQ} := \fs_{P}^{-1}(O_{Y_P}) \times_{Y_P} Y_Q,
$$
where, by abuse of notation, we have denoted by $\fs_P: Y_P \times_Y \Sing(Y) \to \Sing(Y_P)$ the singular codifferential of the map $Y_P \to Y$.
Tautologically, $N_{\PQ}$ is the base-change along $Y\to \LSG$ of the moduli stack $\N_{\PQ}$.
The same procedure as above yields the $\Dmod(\Sing(Y))$-linear functor
$$
\mu_Y^!: \Dmod(N) \longto 
 \lim_{[\PQ] \in \Tw^\op}
\Dmod(N_{\PQ}).
$$
equipped with a $\Dmod(\Sing(Y))$-linear left adjoint. A remark parallel to Remark \ref{rem: passage to left adjoints} applies here: $\mu_Y^!$ admits a $\Dmod(\Sing(Y))$-linear left adjoint, $(\mu_Y)_!$.

\begin{thm} [Strong microlocal gluing] \label{mainthm:microlocal}
The functor $\mu_Y^!$ is an equivalence.
\end{thm}

\begin{rem}
This theorem implies that the functor $\mu^!$ of \eqref{eqn: mu pullback for global nilp cone} is an equivalence. For that, it suffices to prove that the natural maps $\mu_! \mu^! \to \id $ and $\id \to \mu^! \mu_!$ are isomorphisms. This can be checked after base-changing along $Y \tto \LSG$. Thanks to the fact that $\mu$ is pseudo-proper (see \ref{def:pseudo-proper} below), this boils down to the statement of Theorem~\ref{mainthm:microlocal}. 
\end{rem}

\sssec{}

The above theorem, to be proven in in the remainder of this paper, is the main ingredient in the proof of Theorem \ref{mainthm:strong-ICOHN}. Assuming its validity for the time being, let us show how to deduce Theorem \ref{mainthm:strong-ICOHN} from it.

As a preliminary observation, let us note that, since all the maps involved in the definition of $\mu_Y^!$ are $\Gm$-equivariant, it makes sense to consider the functor $(\mu_Y^!)^\Rightarrow$, which is also an equivalence.

\sssec{}

We wish to show that the functor
\begin{equation} 
\gamma^\strong: 
\ICoh_\N(\LSG) 
\longto
\lim_{[\PQ] \in \Tw^\op}
 \ICoh_0((\LSG)^\wedge_{\LSQ})^{P-\temp}
\end{equation} 
of Theorem \ref{mainthm:strong-ICOHN} is an equivalence. To apply the theory developed in Section \ref{sec:global complete intersect} and to connect with Theorem \ref{mainthm:microlocal}, we need to eliminate the stackyness of $\LSG$ and $\LSP$. To this end, recall that the above functor is $\H(\LSG)$-linear (in particular, $\QCoh(\LSG)$-linear) and so it suffices to prove it is an equivalence after pulling back along the atlas $Y \tto \LSG$.

\sssec{}

We have
$$
\QCoh(Y)
\usotimes{\QCoh(\LSG)}
\ICoh_\N(\LSG) 
\simeq
\ICoh_{N}(Y)
$$
and, by Section \ref{sssec:intermediate cat more familiar},
$$
\QCoh(Y)
\usotimes{\QCoh(\LSG)}
\ICoh_0((\LSG)^\wedge_{\LSQ})^{P-\temp}
\simeq
\ICoh_0(Y^\wedge_{Y_P})
\usotimes{\Dmod(Y_P)} \Dmod(Y_Q).
$$
Thus, it is enough to prove that the resulting functor
\begin{equation} \label{eqn:functor gamma Y strong}
\gamma_Y^\strong: 
\ICoh_N(Y)
\longto
\lim_{[\PQ]
 \in \Tw^\op}
\Bigt{
\ICoh_0(Y^\wedge_{Y_P})
\usotimes{\Dmod(Y_P)} \Dmod(Y_Q)
},
\end{equation}
is an equivalence.

\sssec{}

Since $Y$ is a global complete intersection, Section \ref{ssec:global complete intersect} yields an action of $\DSingY$ on $\ICoh(Y)$. Then
$$
\ICoh_N(Y)
\simeq
\ICoh(Y)
\usotimes{
\DSingY
}
\Dmod(N)^\unshift
$$ 
and, by Proposition \ref{prop:ICohzero via DSing},
\begin{eqnarray}
\nonumber
\ICoh_0(Y^\wedge_{Y_P})
\usotimes{\Dmod(Y_P)} \Dmod(Y_Q)
& \simeq &
\ICoh(Y)
\usotimes{\DSingY} \Dmod
\Bigt{ 
\fs_{P}^{-1}(O_{Y_P}) \times_{Y_P} Y_Q
}^\unshift
\\
\nonumber
& = &
\ICoh(Y)
\usotimes{\DSingY} \Dmod
(
N_{\PQ}
)^\unshift.
\end{eqnarray}
These equivalences and the compatibilities of Section \ref{ssec:singular codiff}
imply that the functor $\gamma_Y^\strong$ is obtained from $(\mu_Y^!)^\Rightarrow$ by tensoring up with $\ICoh(Y)$ over $\DSingY$.
It follows that $\gamma_Y^\strong$ is an equivalence.

\ssec{Two key tools} \label{ssec:key-tools}

In this section, we state and prove two results that will be needed
for the proof of Theorem \ref{mainthm:microlocal}. 
The first one is the fact that, under some properness conditions, equivalences of DG categories of $\fD$-modules can be checked on fibers. The second one, which is related to the first, is excision for $\fD$-modules.

\sssec{}\label{def:pseudo-proper}

Let $f: \X \to Y$ be a map of prestacks with $Y$ a scheme. We say that $f$ is \emph{pseudo-proper} if 
$$
\X \simeq \uscolim{a \in A} \, X_a,
$$
where each $f_a: X_a \to Y$ is a scheme proper over $Y$. As pointed out in Remark \ref{rem: passage to left adjoints}, whenever $f$ is pseudo-proper, the pullback functor $f^!: \Dmod(Y) \to \Dmod(\X)$ admits a left adjoint (to be denoted $f_!$).

\sssec{}

Recall from \cite{Crys} that DG categories of $\fD$-modules are defined only for $\Sch^{\on{lft}}_{\kk}$, the $1$-category of those (classical) schemes that are locally of finite type over the ground field $\kk$. Since we are going to consider field extensions, let us be more explicit with the notation: we denote by $\Dmod_{/\kk}: (\Sch^{\on{lft}}_{\kk})^\op \to \DGCat_\kk$ what has been denoted simply by $\Dmod$ throughout.
We are going to also consider the functor $\Dmod_{/\kk'}: (\Sch^{\on{lft}}_{\kk'})^\op \to \DGCat_{\kk'}$, with $\kk'$ a field extension of $\kk$. 
We have: 
$$
\Dmod_{/\kk'}(\Spec(\kk') \times_{\Spec(\kk)} Y)
\simeq
\Vect_{\kk'} 
\otimes _{\Vect_{\kk}}
\, \Dmod_{/\kk}(Y).
$$

\sssec{}

For a $\kk'$-point $y$ of $Y$, we denote by $(i_y)^!$ the functor
$$
\Dmod_{/\kk}(Y)
\to
\Vect_{\kk'} 
\otimes _{\Vect_{\kk}}
\Dmod_{/\kk}(Y)
\simeq
\Dmod_{/\kk'}
\bigt{
Y \times_{\Spec (\kk)} \Spec(\kk')
}
\longto
\Dmod_{/\kk'}(\Spec(\kk')) \simeq \Vect_{/\kk'},
$$ 
where the right arrow is the $!$-pullback along the finite type map $\Spec(\kk') \to \Spec(\kk') \times_{\Spec(\kk)} Y$ induced by $y$.
Alternatively, the same functor $(i_y)^!: \Dmod_{/\kk}(Y) \to \Vect_{\kk'}$ can be rewritten as
$$
\Dmod_{/\kk}(Y)  
\xto{\oblv_Y} 
\QCoh(Y)
\xto{(i_y)^*}
\Vect_{\kk'}.
$$
The following fact, which can be proven by Noetherian induction, was used in \cite[Lemma 6.1.7]{AG2}; it will be crucial also in the proof of Proposition \ref{prop:pseudo-proper-schemes} below.

\begin{lem} \label{lem:conservative system}
Let $Y$ be a scheme locally of finite type over $\kk$. The functors $(i_y)^!$, for all extensions $\kk' \supseteq \kk$ and all $y \in Y(\kk')$, form a conservative system: an object $\F \in \Dmod_{/\kk}(Y)$ is zero iff the objects $(i_y)^!(\F)$ are all zero. 
\end{lem}

\begin{prop} \label{prop:pseudo-proper-schemes}
Let $f: \X = \colim_{a \in A} X_a \to Y$ be a pseudo-proper map. Then $f^!: \Dmod(Y) \to \Dmod(\X)$ is an equivalence if and only if, for any field extension $\kk \subseteq \kk'$ and any $\kk'$-point $y \in Y$, the pullback 
$$
(\restr {f}y)^!:
\Vect_{\kk'}
\longto
\Dmod_{/
\kk'}
\bigt{
\X \times_Y \Spec(\kk')
}
$$ 
along $\X \times_Y \Spec \kk' \to \Spec \kk'$ is an equivalence.
\end{prop}

\begin{proof}
The \virg{only if} direction is obvious. To prove the converse, let us assume that each functor $(\restr {f}y)^!$ is an equivalence. We need to prove that the two natural transformations 
$$
\id_{\Dmod(\X)} \longto f^! \circ f_!
$$
$$
f_! \circ f^! \longto \id_{\Dmod(Y)}
$$
are equivalences. The latter has been dealt with in \cite[Lemma 6.1.7]{AG2} using Lemma \ref{lem:conservative system}, so let us focus on the former. It suffices to show that, for any $M=\{M_a\} \in \Dmod(\X)$ and any $b \in A$, the resulting map
$$
M_b \longto (f_b)^!  \circ f_!(M)
$$ 
is an isomorphism in $\Dmod(X_b)$. We will prove that the arrow
\begin{equation}\label{eqn:topolino}
(i_x)^!(M_b) \longto (i_x)^! \circ (f_b)^! \circ f_!(M)
\end{equation}
is an isomorphism in $\Vect_{\kk'}$ for any $\kk'$-point $i_x: \Spec(\kk') \to X_b$. This is enough in view of the above lemma. Consider the following diagram (with cartesian outer rectangle):
\begin{equation} 
\nonumber
\begin{tikzpicture}[scale=1.5]
\node (00) at (0,0) {$\Spec(\kk')$};
\node (10) at (1.5,0) {$X_b$};
\node (01) at (0,1) {$ \X_y$};
\node (11) at (3,1) {$\X$};
\node (20) at (3,0) {$Y$,};
\path[->,font=\scriptsize,>=angle 90]
(00.east) edge node[above] {$i_x$} (10.west); 
\path[->,font=\scriptsize,>=angle 90]
(01.east) edge node[above] {$\eta$} (11.west); 
\path[->,font=\scriptsize,>=angle 90]
(10.east) edge node[above] {$f_b$} (20.west); 
\path[->,font=\scriptsize,>=angle 90]
(01.south) edge node[right] {$\restr f y$ } (00.north);
\path[<-,font=\scriptsize,>=angle 90]
(11.south west) edge node[right] {$\;\;\iota_b$} (10.north east);
\path[->,font=\scriptsize,>=angle 90]
(11.south) edge node[right] {$f$ } (20.north);
\end{tikzpicture}
\end{equation}
where $\iota_b: X_b \to \X$ is the structure map and $y := f_b(x) \in Y(\kk')$.
Denote by $i_y$ the composition of the two lower horizontal arrows. By extending scalars from $\kk$ to $\kk'$ and renaming $\kk'$ with $\kk$, we can assume that $x$ is a $\kk$-point.
Since the outer rectangle is cartesian and $f$ is pseudo-proper, $f_!$ satisfies base-change against $(i_y)^!$.
We obtain:
$$
(i_x)^! \circ (f_b)^! \circ f_!(M)
\simeq
(\restr f y)_! \circ \eta^!(M).
$$
Note that the point $x \in X_b(\kk)$ yields a section $\wt i_x$ of $\restr f y$. Since $(\restr f y)_!$ is an equivalence by assumption, it follows that $(\restr f y)_! \simeq (\wt i_x)^!$.
Thus, the RHS of \eqref{eqn:topolino} is isomorphic to $(\wt i_x)^! \eta^! (M)$, which is in turn is manifestly isomorphic to $(i_x)^!(M_b)$.
\end{proof}

\sssec{}

Let us put the above result in context.
Set 
$$
N_\Glued := \uscolim{[\PQ] \in \Tw} N_\PQ
$$
and consider the pseudo-proper map $\mu_Y: N_\Glued \to N$ obtained by combining the various maps $N_{\PQ} \to N$. 
Theorem \ref{mainthm:microlocal} states that the pullback
\begin{equation} \label{eqn:glued-mu-equivalence}
\mu_Y^! :\Dmod(N) \to 
\Dmod(N_\Glued)
\end{equation}
is an equivalence. By the above proposition, this can be checked separately on field valued points of $N$ (and hence on field valued points of $\N$):

\begin{cor} \label{cor:can do pointwise}
For $\kk' \supseteq \kk$ a field extension and $(\sigma,A) \in \N(\kk')$ a point, consider the fiber
$$\N_\Glued^{\, \sigma,A} := \N_\Glued \times_\N \{(\sigma,A) \}.$$
If the $!$-pullback
\begin{equation} \label{eqn:glued-mu-equivalence-at-x}
\Vect_{\kk'}
\longto 
\Dmod_{/\kk'}(\N_\Glued^{ \,\sigma,A})
\end{equation}
is an equivalence for all $\kk'$ and all $(\sigma,A) \in \N(\kk')$, then (\ref{eqn:glued-mu-equivalence}) is an equivalence.
\end{cor}

\begin{example} \label{example trivial loc sys}
In the case $\sigma$ is the trivial $G$-local system, we obtain the statement of Theorem \ref{thm:D-mod on Nilp local}.
\end{example}

\sssec{} \label{sssec:beginning of definition of Spr glued unip}

When checking that \eqref{eqn:glued-mu-equivalence-at-x} is an equivalence for a given $\kk'$-point of $\N$, we can replace $\N$ with its extension to $\kk'$ and then rename $\kk'$ by $\kk$. Thus, to prove Theorem \ref{mainthm:microlocal}, it suffices to check that the arrow \eqref{eqn:glued-mu-equivalence-at-x} is an equivalence for any $\kk$-point $(\sigma, A)$. Hereafter, the pair $(\sigma, A)$ will always denote a $\kk$-point of $\N$.

\sssec{}

Observe that the prestack $\N_\Glued^{\, \sigma,A}$ already appeared in \cite[Sections 7.1.4-7.1.7]{AG2} under the name of $\Spr_{\Glued, \unip}^{\, \sigma,A}$, a name that we adopt from now on.
Accordingly, we also set
$$
\Spr_{\PQ, \unip}^\sigmaA
:=
\N_{\PQ} \ustimes \N \{(\sigma, A)\}.
$$
This is the scheme of $Q$-reductions of $\sigma$ with the property that $A$ is a horizontal section of $(\u_P)_{\sigma_Q}$.
Like $\N_{\PQ}$, the assignment
\begin{equation} \label{eqn:pippa}
\Spr^{\sigmaA}_{\Tw, \unip}:
[\PQ] \squigto \Spr_{\PQ, \unip}^\sigmaA
\end{equation}
is a functor out of $\Tw := \Tw(\Par)$.

\begin{rem} \label{rem:strings vs twisted arrows}
Actually, the prestack $\Spr_{\Glued, \unip}^{\, \sigma,A}$ of \cite[Sections 7.1.4-7.1.7]{AG2} was defined slightly differently. Namely, it was defined as the colimit of the composition
$$
\String(\Par) \xto{\beta} \Tw(\Par)
\xto{\eqref{eqn:pippa}} \PreStk,
$$
where $\String(\Par)$ is the $1$-category of strings of parabolics defined in \cite[Section 7.1.2]{AG2}, and $\beta$ is the natural map that retains only the first and last element of a string. 
It is clear that $\beta$ is cofinal, so our definition of $\Spr_{\Glued, \unip}^{\, \sigma,A}$ agrees with the one of \cite[Sections 7.1.4-7.1.7]{AG2}.
\end{rem}

\sssec{}

In \cite{AG2}, it is shown that
\begin{equation} \label{eqn:springer-glued-unipotent}
\Spr_{\Glued, \unip}^{\, \sigma,A}
:=
\uscolim{[\PQ] \in \Tw} \Spr_{\PQ, \unip}^\sigmaA
\end{equation}
is homologically contractible, that is, that the pullback functor
$$
\Vect_\kk
 \longto
 \Dmod(\Spr_{\Glued, \unip}^{\, \sigma,A})
$$
is \emph{fully faithful}. To prove our main theorem, we need to show more: we need to show that the same functor is an equivalence. We will do so in Section \ref{sec:proof} by revisiting the proof of the fully faithfulness given in \cite{AG2}. Essentially, the only improvement to be made is the usage of the full force of the following \emph{excision} result. (The authors of \cite{AG2} only use half of the statement.)

\begin{prop} [$\fD$-module excision] \label{prop:excision} 
Let $f: Y \to Y'$ be a proper morphism of schemes. Let $\iota: Y_0 \hto  Y$ and $\iota': Y'_0 \hto Y'$ be two closed embeddings with the property that $f(Y_0) \subseteq Y'_0$.
If $f$ yields an isomorphism between the two complementary open subschemes, then the natural !-pullback functor
$$
\alpha = (f^!, \on{id}_{\Vect}) :
\Dmod(Y')
\ustimes{\Dmod(Y'_0)}
\Vect
\longto
\Dmod(Y)
\ustimes{\Dmod(Y_0)}
\Vect
$$
is an equivalence.
\end{prop}

\begin{proof}
The properness of $f$ implies that $\alpha$ admits a left adjoint, $\alpha^L$, which sends $(\F,V)$ to $(\G,V)$, where $\G$ is the colimit of
$$
\begin{tikzpicture}[scale=1.5]
\node (00) at (0,0) {$(\iota')_! (\omega_{Y'_0}) \otimes V$};
\node (01) at (0,1) {$(\iota')_! ((f_0)_!\omega_{Y_0}) \otimes V \simeq f_!(\iota_! \iota^! \F)$};
\node (11) at (3,1) {$f_!(\F)$};
\path[->,font=\scriptsize,>=angle 90]
(01.east) edge node[above] {$ $} (11.west); 
\path[->,font=\scriptsize,>=angle 90]
(01.south) edge node[right] {$ $} (00.north);
\end{tikzpicture}
$$
in $\Dmod(Y')$.
To prove that $\alpha$ is an equivalence, we apply the Barr-Beck-Lurie theorem: since $\alpha$ is clearly (continuous and) conservative, it remains to show that $\alpha^L$ is fully faithful.
Denote by $j: U := (Y-Y_0) \hto Y$ the open embedding and by $\C$ the DG category $\Dmod(Y)
\times_{\Dmod(Y_0)}\Vect$. The notations $j'$, $U'$ and $\C'$ bear the obvious parallel meaning.
We observe that $\C$ sits in the exact sequence of DG categories
\begin{equation} 	\nonumber
\begin{tikzpicture}[scale=1.5]
\node (a) at (0,1) {$\Vect$};
\node (b) at (1.4,1) {$\C$};
\node (c) at (3,1) {$\Dmod(Y -Y_0)$,};
\path[right hook ->,font=\scriptsize,>=angle 90]
([yshift= 1.5pt]a.east) edge node[above] {$ \i$} ([yshift= 1.5pt]b.west);
\path[->>,font=\scriptsize,>=angle 90]
([yshift= -1.5pt]b.west) edge node[below] {$ \i^R$} ([yshift= -1.5pt]a.east);
\path[->>,font=\scriptsize,>=angle 90]
([yshift= 1.5pt]b.east) edge node[above] {$ \p$} ([yshift= 1.5pt]c.west);
\path[right hook ->,font=\scriptsize,>=angle 90]
([yshift= -1.5pt]c.west) edge node[below] {$\p^R$} ([yshift= -1.5pt]b.east);
\end{tikzpicture}
\end{equation}
with functors defined as follows:
\begin{eqnarray}
\nonumber
& & \i: V \squigto (\iota_*(V \otimes \omega_{Y_0}), V); 
\\
\nonumber
& & \i^R: (\F, W) \squigto W; 
\\
\nonumber
& & \p: (\F, W) \squigto j^!(\F); 
\\
\nonumber
& & \p^R : \G \squigto (j_*(\G), 0).
\end{eqnarray}
Consider also the same exact sequence for $\C'$. By the assumptions, $\alpha: \C' \to \C$ extends to a functor out of exact sequences (that is, the four squares commute), with the two outer terms being equivalences.
To prove that $\alpha^L$ is fully faithful, it suffices to check that the two natural arrows
$$
\i \longto \alpha \circ \alpha^L \circ \i,
\hspace{.6cm}
\p^R \longto \alpha\circ \alpha^L \circ \p^R
$$
are isomorphisms. Both claims are clear by inspection.
\end{proof}

\begin{rem}
The DG category $\Dmod(Y) \times_{\Dmod(Y_0)} \Vect$ may be regarded as the DG category of $\fD$-modules on the one-point compactification of $Y-Y_0$.
\end{rem}

\sec{Contractibility of glued Springer fibers} \label{sec:proof}

In this final section, we prove Theorem \ref{mainthm:microlocal} and consequently Theorem \ref{mainthm:strong-ICOHN}. We follow \cite [Sections~7-8]{AG2} very closely.

\ssec{Springer fibers}

Recall  $\Spr_{\Glued, \unip}^{\, \sigma,A}$, the prestack of \emph{glued Springer fibers},  defined in Section \ref{sssec:beginning of definition of Spr glued unip} and in (\ref{eqn:springer-glued-unipotent}).
By Corollary \ref{cor:can do pointwise}, it is enough to prove the following result, which is our goal until the end of the paper.

\begin{thm} \label{thm:contractibility-of-twisted-prestack}
For any $(\sigma, A) \in \N(\kk)$, the pullback functor
$$
\Vect_{\kk} 
\longto
\Dmod(\Spr_{\Glued, \unip}^{\, \sigma,A})
$$
is an equivalence.
\end{thm}

\begin{example} \label{example: A=0}
If $A=0$, it is easy to verify that the prestack $\Spr_{\Glued, \unip}^{\, \sigma,A}$ is isomorphic to $\pt = \Spec(\kk)$. The argument appears in \cite[Remark 7.1.10]{AG2}.
Hence, it suffices to treat the case $A \neq 0$, which we assume from now on. Note that $A \neq 0$ implies that $\Spr_{Q \subseteq G, \unip}^\sigmaA \simeq \emptyset$ for any $Q \in \Par$.
\end{example}

\begin{example}
Let $\sigma$ be the trivial local system, so that $A$ is just a nilpotent element of $\g$. If $A$ is regular, then it is elementary to verify that (at the level of reduced schemes)
$$
\Spr_{\PQ,\unip}^{\sigma, A} 
\simeq
\begin{cases} 
\emptyset & \mbox{if } P \neq B \\
\pt  & \mbox{if } P=Q=B,
 \end{cases}
$$
the latter statement corresponding to the well-known fact that a regular nilpotent element is contained in exactly one Borel subgroup. Hence, $\Spr_{\Glued, \unip}^{\, \sigma,A} \simeq \Spr_{B \subseteq B,\unip}^{\sigma, A} \simeq \pt$.
\end{example}

\begin{example}
Let $G=\mathrm{GL}_3$ and $\sigma$ the trivial local system as before, but this time we take $A$ to be a subregular nilpotent element. Denoting by $P_1$ and $P_2$ the two maximal parabolic subgroups, we have (again at the level of reduced schemes):
$$
\Spr_{\PQ,\unip}^{\sigma, A} 
\simeq
\begin{cases} 
\emptyset & \mbox{for } [\mbox{any} \subseteq G] \\
\pt & \mbox{for $[P_1 \subseteq P_1]$ and $[P_2 \subseteq P_2]$} \\
\PP^1 & \mbox{for $[B \subseteq P_1]$ and $[B \subseteq P_2]$} \\
\PP^1 \vee \PP^1  & \mbox{for } [B \subseteq B],
 \end{cases}
$$
so that $\Spr_{\Glued, \unip}^\sigmaA$ is the colimit of the diagram
$$
\pt
\twoheadleftarrow
\PP^1
\hto
\PP^1 \vee \PP^1
\hookleftarrow
\PP^1
\tto
\pt.
$$
Note that $\PP^1 \vee \PP^1$ is the \emph{scheme pushout} of the correspondence $\PP^1 \hookleftarrow \pt \hto \PP^1$, which is different\footnote{This follows from \cite[Chapter III.1, Remark 1.4.3]{Book}.} from the \emph{prestack pushout} of the same diagram. However, reasoning as in Proposition \ref{prop:pseudo-proper-schemes}, the natural map from the latter to the former induces an equivalence on $\fD$-modules. In particular, we have
$$
\Dmod(\PP^1 \vee \PP^1)
\simeq
\Dmod(\PP^1) \ustimes{\Dmod(\pt)} \Dmod(\PP^1),
$$
so that $\Dmod(\Spr_{\Glued, \unip}^\sigmaA)$ is equivalent to the DG category of $\fD$-modules on the prestack colimit of the diagram
$$
\pt
\twoheadleftarrow
\PP^1
\hookleftarrow
\pt
\hto
\PP^1
\tto
\pt.
$$
The latter colimit is evidently isomorphic to $\pt$.

\end{example}

\sssec{}

For any $P \in \Par$, recall the classical scheme $\Spr_P^{\sigmaA}$ that parametrizes $P$-reductions $\sigma_P$ of $\sigma$ with the property that $A$ belongs to $H^0(X_\dR, \p_{\sigma_P})$. 
Let $\Par' := \Par - \{G\}$ be the poset of proper parabolics and form the prestack
$$
\Spr_\Glued^\sigmaA
:=
\uscolim{R \in \Par'} \Spr_R^\sigmaA.
$$

\sssec{}

We shall relate the latter to our $\Spr_{\Glued, \unip}^\sigmaA$. To this end, consider the poset $\TwTr'$ of \emph{twisted triples} in $\Par'$: an object of $\TwTr'$ is a triple $[\PQ \subseteq R]$ of proper parabolics, a morphism from such triple to $[Q' \subseteq P' \subseteq Q' ]$ exists if and only if 
$$
Q \subseteq Q' \subseteq P' \subseteq P \subseteq R  \subseteq R'.
$$
We have functors
$$
\phi_1 :\TwTr' \to \Tw(\Par'), 
\hspace{.4cm}
[\PQ \subseteq R] \squigto [\PQ],
$$
$$
\phi_2 : \TwTr' \to \Par', 
\hspace{.4cm}
[\PQ \subseteq R] \squigto R.
$$

\sssec{}

Setting
$$
\Spr_{\Glued, \mixed}^\sigmaA :=
\colim \, \Spr^{\sigmaA}_{\Tw, \unip} \circ \phi_1
\simeq
\uscolim{ [\PQ \subseteq R] \in \TwTr'} \Spr^{\sigmaA}_{\PQ, \unip}
,
$$
we have the natural correspondence
\begin{equation} \label{eqn:various springers}
\Spr_{\Glued}^\sigmaA
\longleftarrow
\Spr_{\Glued, \mixed}^\sigmaA
\longto
\Spr_{\Glued, \unip}^\sigmaA.
\end{equation}

\sssec{}

In \cite[Section 7.3.4]{AG2}, it is proven\footnote{One needs to take into account Remark \ref{rem:strings vs twisted arrows}.} that the right arrow of \eqref{eqn:various springers} is an isomorphism of prestacks. 
Hence, Theorem \ref{thm:contractibility-of-twisted-prestack} is a consequence of the two results below.

\begin{prop} \label{prop:from-mixed-to-simpler-gluing}
Assume that Theorem \ref{thm:contractibility-of-twisted-prestack} holds true for all proper Levi's of $G$. For $A \neq 0$, the pullback functor
$$
\Dmod(\Spr_{\Glued}^\sigmaA)
\longto
\Dmod(\Spr_{\Glued, \mixed}^\sigmaA)
$$
is an equivalence.
\end{prop}

\begin{thm} \label{thm:conctractibility of easier gluing}
For $A \neq 0$, the pullback functor $\Vect \to \Dmod(\Spr_\Glued^\sigmaA)$ is an equivalence.
\end{thm}

\begin{rem}
In both statements above, we have imposed the condition $A \neq 0$. This restriction is harmless in view of Example \ref{example: A=0}, which takes care of the case $A=0$.
\end{rem}

\sssec{}

The proof of Theorem \ref{thm:conctractibility of easier gluing} is deferred to Section \ref{ssec:weyl combinatorics}. On the other hand, Proposition \ref{prop:from-mixed-to-simpler-gluing} is an immediate consequence of the following lemma.

\begin{lem}
For any $R \in \Par'$, consider the subcategory $\Tw_R \subseteq \Tw$ consisting of pairs $[\PQ]$ with $P \subseteq R$. Then pullback along the natural functor
$$
\uscolim{[\PQ] \in \Tw_R} 
\Spr^{\sigmaA}_{\PQ, \unip}
\longto
\Spr_R^\sigmaA
$$
yields an equivalence at the level of $\fD$-modules.
\end{lem}

\begin{proof}
Since the map in question is pseudo-proper, Proposition \ref{prop:pseudo-proper-schemes} ensures that it suffices to check the statement at the level of field-valued points of $\Spr_R^\sigmaA$. That is, for $\sigma_R \in \Spr_R^\sigmaA$, we need to show that the pullback 
$$
\Vect \longto
\Dmod
\Bigt{
\uscolim{[\PQ] \in \Tw_R} 
\Spr^{\sigmaA}_{\PQ, \unip}
\ustimes{
\Spr_R^\sigmaA
}
\sigma_R
}
$$
is an equivalence. We will use the assumption of Proposition \ref{prop:from-mixed-to-simpler-gluing} on $M$, the Levi quotient of $R$. Let $\sigma_M$ to be the $M$-local system induced from $\sigma_R$ and $A_M$ the projection of $A$ to a section of $\fm_{\sigma_M}$. As in \cite[Section 7.4.3]{AG2}, we observe that there is a tautological isomorphism
$$
\uscolim{[\PQ] \in \Tw_R} 
\Spr^{\sigmaA}_{\PQ, \unip}
\ustimes{
\Spr_R^\sigmaA
}
\sigma_R
\simeq
\Spr_{\Glued, \unip}^{\sigma_M,A_M}
$$
of prestacks.
\end{proof}

\ssec{Weyl combinatorics} \label{ssec:weyl combinatorics}

It remains to prove Theorem \ref{thm:conctractibility of easier gluing}, a task that will keep us busy until the end of the paper. We need some Weyl combinatorics: we refer to \cite[Sections 8.1-8.3]{AG2} for details and more on the notation.

\sssec{}

Let $W$ denote the Weyl group and $G$ and $I$ the set of nodes of the Dynkin diagram.
Let $P_0$ be the standard parabolic associated to $A$ by the Jacobson-Morozov theorem: this makes sense because $A \neq 0$ by assumption. Denoting by $J_0 \subset I$ the subset corresponding to $P_0$, define
$$
W' :=
\{
w \in W \, : \, 
w^{-1}(J_0) \subseteq \mathsf{R}^{+}
\}.
$$
By \cite[Section 8.2.1]{AG2}, $W'$ contains a unique maximal element, denoted by $w_0'$. 

\sssec{}

Elements of $W'$ measure the relative position of $P_0$ and other parabolics. 
In particular, each $\Spr_P^\sigmaA$ is stratified by elements of $W'$. Thus, we have closed embeddings
$$
\Spr_P^{\sigma, A, < w}
\hto
\Spr_P^{\sigma, A, \leq w}
\hto
\Spr_P^\sigmaA
$$
that are moreover functorial in $P$.
Hence, as in \cite[Section 8.3.3]{AG2}, the $W'$-stratification on each $\Spr_P^\sigmaA$  induces a stratification of $\Spr_\Glued^{\sigmaA}$ by letting
$$
\Spr_\Glued^{\sigma, A, < w} 
:= 
\uscolim {P \in \Par'}  \Spr_P^{\sigma, A, < w}, 
\hspace{.4cm}
\Spr_\Glued^{\sigma, A, \leq w} 
:= 
\uscolim {P \in \Par'}  \Spr_P^{\sigma, A, \leq w}.
$$

\sssec{}

Since $W'$ contains a maximal element $w_0'$, the following proposition is enough to prove Theorem \ref{thm:conctractibility of easier gluing}.

\begin{prop}
The pullback functor
$$
\Vect \longto \Dmod (\Spr_\Glued^{\sigma, A, \leq w} )
$$
is an equivalence for any $w \in W'$. 
\end{prop}

\begin{proof}
We proceed by induction on the length of $w$. The base case of $w =1$ is settled by \cite[Section~8.4.1]{AG2}, where it is proven that $\Spr_\Glued^{\sigma, A, \leq 1} $ is isomorphic to $\pt$.

\medskip

Next, let $w \neq 1$ be fixed, and assume that the assertion holds for any $w' < w$. Observe first that the pullback functor
$$
\Vect \longto \Dmod(\Spr_\Glued^{\sigma, A, < w})
$$
is an equivalence. Indeed, we tautologically have
$$
\Spr_\Glued^{\sigma, A, < w}
\simeq
\uscolim{w' < w}
\Spr_\Glued^{\sigma, A, \leq w'},
$$
so our claim follows by the induction hypothesis together with the fact that the poset indexing the colimit is contractible (as it contains the minimum element $1$).

\medskip

Consider now the prestacks
$$
\Spr_P^{\sigma, A, \leq w}/\Spr_P^{\sigma, A, < w} 
:=
\Spr_P^{\sigma, A, \leq w} 
\underset{\Spr_P^{\sigma, A, < w} }\sqcup \pt
$$
$$
\Spr_{\Glued}^{\sigma, A, \leq w}/\Spr_{\Glued}^{\sigma, A, < w} 
:=
\Spr_{\Glued}^{\sigma, A, \leq w} 
\underset{\Spr_{\Glued}^{\sigma, A, < w} }\sqcup \pt
$$
and the resulting fiber square
\begin{equation} 
\nonumber
\begin{tikzpicture}[scale=1.5]
\node (00) at (0,0) {$  \Vect $};
\node (10) at (3.2,0) {$  \Dmod(\Spr_{\Glued}^{\sigma, A, < w}) $};
\node (01) at (0,1) {$ \Dmod \bigt{ \Spr_\Glued^{\sigma, A, \leq w}/\Spr_\Glued^{\sigma, A, < w} }$};
\node (11) at (3.2,1) {$  \Dmod(\Spr_\Glued^{\sigma, A, \leq w} )$};
\path[->,font=\scriptsize,>=angle 90]
(00.east) edge node[above] {pullback} (10.west); 
\path[->,font=\scriptsize,>=angle 90]
(01.east) edge node[above] {pullback} (11.west); 
\path[->,font=\scriptsize,>=angle 90]
(01.south) edge node[right] {pullback} (00.north);
\path[->,font=\scriptsize,>=angle 90]
(11.south) edge node[right] {pullback} (10.north);
\end{tikzpicture}
\end{equation}
of DG categories. By the above observation, we know that the bottom horizontal arrow is an equivalence, hence so is the top one.
Thus, it remains to show that the left vertical arrow is an equivalence, which is the content of the next lemma.
\end{proof}

\begin{lem} \label{lem:technical-weyl-groups}
For $w \neq 1$, the pullback functor
$$
\Vect 
\longto 
\Dmod \bigt{ \Spr_\Glued^{\sigma, A, \leq w}/\Spr_\Glued^{\sigma, A, < w} }
$$
is an equivalence.
\end{lem}

\begin{proof}
We proceed in 5 steps.

\sssec*{Step 1.}
By  \cite[Section 8.5.6]{AG2}, the prestack
$$
\Spr_\Glued^{\sigma, A, \leq w_0'}/\Spr_\Glued^{\sigma, A, < w_0'} 
$$
is isomorphic to $\pt$. Hence, from now on we assume that $w \neq w_0'$.

\sssec*{Step 2.}

Tautologically,
$$
\Dmod \bigt{ \Spr_\Glued^{\sigma, A, \leq w}/\Spr_\Glued^{\sigma, A, < w} }
\simeq
\lim_{P \in (\Par')^\op}
\Dmod \bigt{ \Spr_P^{\sigma, A, \leq w}/\Spr_P^{\sigma, A, < w} }.
$$
Since the pullback maps forming this limit admit left adjoints, we also have
$$
\Dmod \bigt{ \Spr_\Glued^{\sigma, A, \leq w}/\Spr_\Glued^{\sigma, A, < w} }
\simeq
\uscolim{P \in \Par'} \; 
F
$$
where $F: \Par' \to \DGCat$ denotes the functor
$$
P \squigto \Dmod(\Spr_P^{\sigma, A, \leq w}/\Spr_P^{\sigma, A, < w} )
$$
induced by those left adjoints.

\sssec*{Step 3.}

Recall the partition $I = I^0_w \sqcup I^+_w \sqcup I^-_w$, with $I^0_w := I \cap w^{-1}(\mathsf{R}_{J_0})$ and $I^{\pm}_w := I \cap w^{-1}(\mathsf{R}^\pm - \mathsf{R}_{J_0})$. Let
$$
\Par'_w := \{P \in \Par' \; : \; J_P \subseteq  I^0_w \sqcup I^-_w   \}.
$$
The inclusion $\phi: \Par'_w \hto \Par'$ admits a right adjoint $\psi$ that sends $P \squigto \wt P$, where $\wt P$ is the standard parabolic with $J_{\wt P} = J_{P_0} - I_w^+$.

\medskip

According to \cite[Lemma 8.4.3]{AG2}, the tautological functor
\begin{equation} \label{eqn:techincal-counit}
\colim F \circ \phi
\longto
\colim F
\end{equation}
is an equivalence provided that, for any $P \in \Par'$, the natural arrow
$$
F \circ \phi \circ \psi (P) \longto
F(P)
$$
is an equivalence of DG categories. We will prove this statement in the next step.

\sssec*{Step 4.}

We need to show that, for $w \in W'-\{1, w_0'\}$, the pullback functor
$$
f^!:
\Dmod
\bigt{
\Spr_P^{\sigma, A, \leq w}/\Spr_P^{\sigma, A, < w}
}
\longto
\Dmod
\bigt{
\Spr_{\wt P}^{\sigma, A, \leq w}/\Spr_{\wt P}^{\sigma, A, < w}
}
$$ 
along the natural map
$$
f: 
\Spr_{\wt P}^{\sigma, A, \leq w}/\Spr_{\wt P}^{\sigma, A, < w}
\longto
\Spr_P^{\sigma, A, \leq w}/\Spr_P^{\sigma, A, < w}
$$
is an equivalence.
We will prove this by excision, that is, by invoking Proposition \ref{prop:excision}. First off, note that the map
$$
\wt f: 
\Spr_{\wt P}^{\sigma, A, \leq w}
\longto
\Spr_P^{\sigma, A, \leq w}
$$
inducing $f$ is proper. It remains to observe that $\wt f$ restricts to an isomorphism 
$$
\bigt{ 
\Spr_{\wt P}^{\sigma, A, \leq w} - \Spr_{\wt P}^{\sigma, A, < w}
}
\longto
\bigt{ 
\Spr_P^{\sigma, A, \leq w} - \Spr_P^{\sigma, A, < w}
}
$$
on the open complements. This follows from \cite[Lemma 8.2.10(2)]{AG2}, which states that the scheme of $P$-flags in position $w$ with a given $P_0$-flag is isomorphic to the scheme of $\wt P$-flags in position $w$ with the given $P_0$-flag.

\sssec*{Step 5.}

The equivalence (\ref{eqn:techincal-counit}) says (by passing to right adjoints) that the pullback functor
$$
\Dmod \bigt{ \Spr_\Glued^{\sigma, A, \leq w}/\Spr_\Glued^{\sigma, A, < w} }
\longto
\Dmod \bigt{ 
\uscolim{P \in \Par_w'} \; 
\Spr_P^{\sigma, A, \leq w}/\Spr_P^{\sigma, A, < w} }
$$
is an equivalence. To finish, we invoke \cite[Section 8.5.5]{AG2}, where it is proven that 
the prestack
$$
\uscolim{P \in \Par_w'} \; 
\Spr_P^{\sigma, A, \leq w}/\Spr_P^{\sigma, A, < w}
$$
is isomorphic to $\pt$.
\end{proof}


\end{document}